\documentclass{amsart}
\usepackage{amsmath,amssymb,amsthm,amsfonts,mathrsfs,stmaryrd}
\usepackage[frame,cmtip,arrow,matrix,line,graph,curve]{xy}
\usepackage{graphpap,color,paralist,pstricks}
\usepackage[mathscr]{eucal}
\usepackage{mathabx}
\usepackage[pdftex,colorlinks,backref=page,citecolor=blue]{hyperref}
\usepackage{tikz-cd}
\usepackage{epic,eepic}
\usepackage{yfonts}
\usepackage{enumerate}
\usepackage{ytableau}
\usepackage{comment}

\theoremstyle{definition}
\newtheorem{theorem}{Theorem}[section]
\newtheorem{definition}[theorem]{Definition}
\newtheorem{conjecture}[theorem]{Conjecture}
\newtheorem{lemma}[theorem]{Lemma}
\newtheorem{proposition}[theorem]{Proposition}
\newtheorem{corollary}[theorem]{Corollary}

\newtheorem*{theorem*}{Theorem}

\theoremstyle{remark}
\newtheorem{remark}[theorem]{Remark}
\newtheorem{example}[theorem]{Example}
\def\PP{\mathbb{P}}
\def\RR{\mathbb{R}}
\def\CC{\mathbb{C}}
\def\ZZ{\mathbb{Z}}
\def\QQ{\mathbb{Q}}
\def\K{\mathbb{K}}

\DeclareMathOperator{\tr}{tr}

\DeclareMathOperator{\sgn}{sgn}
\DeclareMathOperator{\rk}{rk}

\DeclareMathOperator{\vol}{vol}
\DeclareMathOperator{\MV}{V}
\DeclareMathOperator{\D}{D}

\DeclareMathOperator{\NS}{NS}
\DeclareMathOperator{\Id}{Id}
\DeclareMathOperator{\End}{End}
\DeclareMathOperator{\bs}{Bs}

\DeclareMathOperator{\psef}{psef}

\DeclareMathOperator{\Psef}{Psef}
\DeclareMathOperator{\Nef}{Nef}
\DeclareMathOperator{\bg}{Big}
\DeclareMathOperator{\Amp}{Amp}

\DeclareMathOperator{\im}{Im}

\title{Correlation inequalities for volume polynomials}
\author{June Huh}
\address{Princeton University and Korea Institute for Advanced Study.}
\email{huh@princeton.edu}

\author{Mateusz Micha{\l}ek}
\address{
University of Konstanz
}
\email{mateusz.michalek@uni-konstanz.de}

\author{Botong Wang}
\address{University of Wisconsin-Madison.}
\email{wang@math.wisc.edu}

\author{Jian Xiao}
\address{
Tsinghua University}
\email{jianxiao@tsinghua.edu.cn}

\begin{document}
\begin{abstract}
We establish sharp correlation inequalities for volume polynomials associated with covers of multisets. These inequalities simultaneously generalize the Khovanskii--Teissier and reverse Khovanskii--Teissier inequalities for intersections of nef divisor classes on projective varieties. They also hold for volume polynomials arising from convex bodies and Hermitian matrices, recovering and extending classical results such as the Alexandrov--Fenchel inequality for mixed volumes of convex bodies and the Hadamard--Fischer inequality for determinants of positive semidefinite matrices. As a combinatorial application, we prove a negative correlation inequality for partitions of bases of matroids that are linearly realizable over a field.
\end{abstract}
\maketitle

\section{Introduction}
Disjoint sets of edges are negatively correlated in a uniformly random spanning tree of a graph. More precisely, let $G$ be a finite connected graph with edge set $E$, and for $S\subseteq E$, let $\Pr(S)$ be the probability that a uniformly random spanning tree of $G$ contains $S$. Then, for any pairwise disjoint subsets $I_1,\dots,I_m\subseteq E$, we have
\[
\Pr(I_1\sqcup I_2\sqcup\cdots\sqcup I_m)
\leq
\Pr(I_1)\Pr(I_2)\cdots \Pr(I_m).
\]
In particular, for any two distinct edges $e_1$ and $e_2$ of $G$, we have the pairwise negative correlation inequality
\[
\Pr(\{e_1,e_2\})
\leq
\Pr(\{e_1\})\Pr(\{e_2\}).
\]
These inequalities have their origins in Kirchhoff's work on electrical networks; see \cite{Pemantle95}.

A \emph{matroid} $M$ on a finite set $E$ is specified by a nonempty collection of subsets of $E$, called the \emph{bases} of $M$, satisfying the \emph{symmetric basis-exchange property}: 
\begin{quote}
\emph{For any bases $B$ and $B'$ of $M$ and any $i\in B\setminus B'$, there is $i'\in B'\setminus B$ such that both $(B\setminus i)\cup i'$ and $(B'\setminus i' )\cup i$ are bases of $M$.}
\end{quote}
A subset of $E$ is \emph{independent} in $M$ if it is contained in a basis of $M$. 
The matroid $M$ is \emph{linearly realizable} over a field if there is a collection of vectors $(v_e)_{e\in E}$ over that field such that a subset $I\subseteq E$ is independent in $M$ if and only if the corresponding set of vectors is linearly independent. For a comprehensive introduction to the theory of matroids, see \cite{oxley2006matroid}.

The spanning trees of a connected graph are precisely the bases of its graphic matroid \cite[Section 1.2]{oxley2006matroid}. We consider more generally a matroid $M$ on $E$, and for $S\subseteq E$, let $\Pr(S)$ be the probability that a uniformly random basis of $M$ contains $S$. As observed by Seymour and Welsh \cite{SW75}, the pairwise negative correlation inequality 
can fail for non-graphic matroids, even for matroids that are linearly realizable over a field; see \cite{HSW22} for more examples and related discussions. Nevertheless, we prove the following negative correlation inequality for partitions of a basis, which is a special case of Theorem~\ref{thm:matroid}.

\begin{theorem}
\label{thm:intro-basis-partition}
Let $M$ be a matroid that is linearly realizable over a field. For any partition $I_1\sqcup\cdots\sqcup I_m$ of a basis $B$ of $M$, we have
\[
\Pr(I_1 \sqcup I_2 \sqcup \cdots \sqcup I_m)
\leq
\Pr(I_1)\Pr(I_2) \cdots \Pr(I_m).
\]
\end{theorem}

For a bipartition $B=I_1\sqcup I_2$, the inequality 
follows directly from the \emph{multiple symmetric basis-exchange property} of Brylawski \cite{Brylawski}, Greene \cite{Greene}, and Woodall \cite{Woodall}:
\begin{quote}
\emph{For any matroid $M$,  any bases $B$ and $B'$ of $M$, and 
any subset $S \subseteq B$, there is a subset $S' \subseteq B'$ such that both $(B \setminus S) \cup S'$ and $(B' \setminus S') \cup S$ are bases of $M$.}
\end{quote}
Applying this property to $I_1 \subseteq B$ and each basis $B'$ of $M$, we obtain an injective map
\[
\{\text{bases of $M$}\} \hookrightarrow \{\text{bases of $M$ containing $I_1$}\} \times \{\text{bases of $M$ containing $I_2$}\}.
\]
Consequently, the bipartition case of Theorem~\ref{thm:intro-basis-partition} holds for every matroid $M$, whether or not it is linearly realizable over a field.
Our proof of Theorem~\ref{thm:intro-basis-partition}  relies on Theorem~\ref{thm introduction} below and applies only to matroids that are linearly realizable over some field. We refer the reader to Section~\ref{sec:matroids} for a generalization of Theorem~\ref{thm:intro-basis-partition} and its conjectural extension to arbitrary matroids.

An inequality of the same form holds for principal minors of positive
semidefinite matrices. Let $A$ be a positive semidefinite
Hermitian matrix, and let $E$ be the set of row and column indices of $A$. For $S\subseteq E$,
let $A_{\widehat S}$ be the principal submatrix of $A$ with rows and columns indexed
by $E\setminus S$. \emph{Koteljanskii's inequality} \cite[Theorem 7.8.9]{HJ} implies that,
for any partition $I_1\sqcup\cdots\sqcup I_m$ of $E$, we have
\begin{equation*}\label{eq:intro-principal-minors}
\det(A)^{m-1}
\leq 
\det(A_{\widehat{I_1}})\det(A_{\widehat{I_2}}) \cdots \det(A_{\widehat{I_m}}).
\end{equation*}
For a bipartition $E=I_1 \sqcup I_2$, this is \emph{Fischer's inequality} \cite[Theorem 7.8.5]{HJ}
\[
\det A
\leq
\det (A_{I_1}) \det (A_{I_2}),
\]
where $A_I$ denotes the principal submatrix of $A$ with rows and columns indexed by $I$. Repeated application of Fischer's inequality yields \emph{Hadamard's inequality}, which states  that the determinant of a positive semidefinite matrix is at most the product of its diagonal entries \cite[Theorem 7.8.1]{HJ}.

The correlation  inequality of Theorem~\ref{thm:intro-basis-partition} also has a counterpart for
volumes of coordinate projections. Let $E$ be a nonempty finite set, and let $K$ be a nonempty compact subset of $\RR^E$.
For $S \subseteq E$, let $K_{\widehat{S}}$ be the image of $K$ under the coordinate projection $\mathbb{R}^E \to \mathbb{R}^{E \setminus S}$. 
Then, for any partition $I_1\sqcup\cdots\sqcup I_m$ of $E$, 
\[
 \vol(K)^{m-1}
 \leq
 \vol(K_{\widehat{I_1}})\vol(K_{\widehat{I_2}}) \cdots \vol(K_{\widehat{I_m}}),
\]
where each volume on the right-hand side 
is taken in the corresponding coordinate subspace. 
This inequality is trivial for bipartitions of $E$, but not for more general partitions. For example, when each part is a singleton, this is the \emph{Loomis-Whitney inequality} \cite[Theorem 1]{LoomisWhitney} 
\[
\vol(K)^{|E|-1}
\leq
\prod_{e \in E} \vol(\pi_{\widehat e} K),
\]
where $\pi_{\widehat e}$ is the coordinate projection that forgets the $e$-th coordinate of $\mathbb{R}^E$.

Let $E$ be a nonempty finite set, and let $r$ be a positive integer. A collection of subsets
$I_1,\dots,I_m$ of $E$ is called an \emph{$r$-cover} of $E$ if every element
of $E$ belongs to exactly $r$ of the subsets $I_j$. 
The inequality for volumes of coordinate projections of $K$ is a special case of the \emph{Bollob\'as--Thomason inequality} \cite[Theorem 2]{bollobas1995projections}:
For every $r$-cover $I_1,\ldots,I_m$ of $E$ 
and every compact subset $K\subseteq\RR^E$, we have
\[
\vol(K)^{m-r}
\leq 
\vol(K_{\widehat{I_1}}) \vol(K_{\widehat{I_2}})  \cdots \vol(K_{\widehat{I_m}}).
\]
Alternatively, applying this result to the $(m-r)$-cover of $E$ formed by the complements of $I_1,\ldots,I_m$ in $E$, we obtain the dual formulation
\[
\vol(K)^{r}
\leq 
\vol(K_{I_1}) \vol(K_{I_2})  \cdots \vol(K_{I_m}),
\]
where $\vol(K_{S})$ denotes the volume of the image of $K$ under the coordinate  projection $\mathbb{R}^E \to \mathbb{R}^S$.

Many fundamental inequalities for intersection numbers in algebraic geometry and complex geometry have a similar form.
 Let $X$ be a $d$-dimensional projective variety over an
algebraically closed field. 
For Cartier divisors
$\gamma_1,\dots,\gamma_d$ on $X$, we denote their intersection number by
\[
(\gamma_1 \gamma_2 \cdots \gamma_d)\coloneq \deg_X (\gamma_1 \gamma_2 \cdots \gamma_d). 
\]
Recall that a Cartier divisor is \emph{nef} if its intersection with
every irreducible curve is nonnegative;  see
\cite[Section 1.4]{lazarsfeldPosI}. All our results remain valid for nef $\mathbb{R}$-divisors in the sense of \cite[Section 1.3]{lazarsfeldPosI}.
When $X$ is a surface, the
\emph{Hodge index theorem} \cite[Exercise V.1.9]{Hartshorne} implies that  any nef divisors $\alpha$ and $\beta$ on $X$ satisfy
\begin{equation*}\label{eq:intro-hodge}
(\beta^2)(\alpha^2)
\leq (\beta\alpha)^2.
\end{equation*} 
In arbitrary dimension, the \emph{reverse Khovanskii--Teissier inequality}
of \cite{lehmann2017correspondences}, in the generalized form established in \cite[Theorem 3.5]{jiang2023algebraic}, states that, for any nef divisors
$\alpha,\beta_1,\dots,\beta_d$ on $X$ and any integer $k$ with
$1\leq k<d$, we have
\begin{equation*}\label{eq:intro-rKT}
\Big(\beta_1\beta_2 \cdots\beta_d\Big)\Big(\alpha^{[d]}\Big)
\leq
\Big(\beta_1\cdots\beta_k\, \alpha^{[d-k]}\Big)
\Big(\beta_{k+1}\cdots\beta_d\, \alpha^{[k]}\Big),
\end{equation*}
where $\alpha^{[k]}\coloneq\alpha^k/k!$ denotes the $k$-th divided power of $\alpha$.
These inequalities also hold in the compact K\"ahler setting, with nef divisor classes replaced by nef  $(1,1)$-classes and intersection numbers replaced by integrals of cup products. For real $(1,1)$-classes $\delta_1,\dots,\delta_d$ on a $d$-dimensional compact K\"ahler manifold $Y$, we use the notation
\[
(\delta_1\delta_2\cdots\delta_d)
\coloneq \int_Y \delta_1\wedge\delta_2\wedge\cdots\wedge\delta_d.
\]
A real $(1,1)$-class is \emph{nef} if it belongs to the closure of the K\"ahler cone of $Y$. With these conventions, the Hodge index inequality and the reverse Khovanskii--Teissier inequality for nef $(1,1)$-classes on compact K\"ahler manifolds take the same form as their counterparts for nef divisors on projective varieties stated above \cite[Section 5]{lehmann2017correspondences}.

Our main result, Theorem~\ref{thm introduction} below, is a common generalization of the Hodge index inequality and the reverse Khovanskii--Teissier inequality. Through the theory of volume polynomials, it also specializes to the correlation inequalities for graphs and matroids stated above, the Hadamard--Fischer inequality for positive semidefinite matrices, as well as the Bollob\'as--Thomason inequality and the Alexandrov--Fenchel inequality for convex bodies. To state the result, we first extend the notion of an $r$-cover from sets to multisets.

Let $E$ be a nonempty finite set. 
A \emph{multiset} on $E$ is a nonnegative integer-valued function
$B:E\to\ZZ_{\geq 0}$. 
We call $B(e)$ the \emph{multiplicity} of $e\in E$ in $B$, and define the \emph{cardinality} of $B$ by
\[
|B|\coloneq \sum_{e\in E} B(e).
\]
We say that a multiset is \emph{empty} if its cardinality is zero.
We identify each subset of $E$ with its indicator function, thereby regarding it as a multiset on $E$. For multisets $B$ and $C$ on $E$, we write $C \subseteq B$ if the multiplicity of $e$ in $C$ is at most the multiplicity of $e$ in $B$ for every $e\in E$.

Let $B$ be a multiset on $E$, and let $r$ be a positive integer. 
An \emph{$r$-cover} of $B$ is a finite
collection of multisets $I=(I_1,\dots,I_m)$ with each $I_j \subseteq B$ such that
\[
\sum_{j=1}^{m} I_j=rB.
\]
Empty members and repetitions are allowed in $I$. When $B$ is a subset of $E$, each $I_j$ is also a subset of $E$, and 
the $r$-cover condition says that every element of $B$ belongs to exactly $r$
members of $I$. 
 If $B$ is nonempty, then  $r$ is at most $m$, with equality precisely when $I_j=B$ for every $j$. 
For an $r$-cover $I=(I_1,\ldots,I_m)$ of $B$, we introduce the constants
\[
\binom{B}{I} \coloneq \binom{B}{I_1} \binom{B}{I_2}\cdots \binom{B}{I_m}, \ \ \text{where} \ \ \binom{B}{I_j}\coloneq \prod_{e\in E}\binom{B(e)}{I_j(e)}, \ \ \text{and} \ \ B! \coloneq \prod_{e \in E} B(e)!.
\]
When $B$ is a subset of $E$, these constants are  $1$.

Let $d$ be a positive integer, and let $X$ be a $d$-dimensional projective variety over an algebraically closed field. 
As before, for a nef divisor $\alpha$ on $X$ and a nonnegative integer $k$, we write $\alpha^{[k]}\coloneq\alpha^k/k!$ for the $k$-th divided power of $\alpha$. For a collection of nef divisors $\beta=(\beta_e)_{e\in E}$ on $X$ and a multiset $B$ on $E$, we set 
\[
\beta^{[B]}\coloneq\prod_{e\in E}\beta_e^{[B(e)]}.
\]
When $B$ is a subset of $E$, this reduces to  the product of $\beta_e$ for $e \in B$.

\begin{theorem}
\label{thm introduction}
Let $I=(I_1,\ldots,I_m)$ be an $r$-cover of a multiset $B$ of cardinality $d$ on $E$. 
For any nef divisor $\alpha$ and any collection of nef divisors $\beta=(\beta_e)_{e \in E}$ on a $d$-dimensional projective variety, we have 
\[
\binom{B}{I}
\left(\beta^{[B]}\right)^r
\left(\alpha^{[d]}\right)^{m-r}
\leq
\left(\beta^{[I_1]}\alpha^{[d-|I_1|]}\right)
\left(\beta^{[I_2]}\alpha^{[d-|I_2|]}\right)
\cdots
\left(\beta^{[I_m]}\alpha^{[d-|I_m|]}\right).
\]
The same inequality holds for any nef $(1,1)$-class $\alpha$ and any collection of nef $(1,1)$-classes $\beta=(\beta_e)_{e \in E}$ on a $d$-dimensional compact K\"ahler manifold.
\end{theorem}

When $E$ has cardinality $1$, we recover the Hodge index inequality from the $1$-cover $(E,E)$ of $2E$. When $E$ has cardinality $d$, we recover the reverse Khovanskii--Teissier inequalities from the bipartitions of $E$.
The constant  $\binom{B}{I}$ is optimal for every
choice of $B$ and $I$; see Example~\ref{ex:optimal}.

We say that an $r$-cover $I$ of $B$ is \emph{indecomposable} if no nonempty subcollection of $I$ is a $q$-cover of $B$ for nonnegative $q<r$. To prove Theorem~\ref{thm introduction}, we may and will consider only the indecomposable covers of $B$, of which there are only finitely many for each fixed $B$. 
For example, suppose that $B$ is a set of cardinality $4$. 
In this case, Shapley enumerated all the indecomposable covers of $B$ in  \cite[Section 7]{Shapley}. 
Up to relabeling, the inequalities associated with nontrivial indecomposable $1$-covers of $B$ are
\begin{align*}
&\Big(\beta_1\beta_2\beta_3\beta_4\Big)\Big(\alpha^{[4]}\Big)^{1}
\le \Big(\beta_1\alpha^{[3]}\Big)\Big(\beta_2\beta_3\beta_4\alpha^{[1]}\Big),\\
&\Big(\beta_1\beta_2\beta_3\beta_4\Big)\Big(\alpha^{[4]}\Big)^1
\le \Big(\beta_1\beta_2\alpha^{[2]}\Big)\Big(\beta_3\beta_4\alpha^{[2]}\Big),\\
&\Big(\beta_1\beta_2\beta_3\beta_4\Big)\Big(\alpha^{[4]}\Big)^2
\le \Big(\beta_1\alpha^{[3]}\Big)\Big(\beta_2\alpha^{[3]}\Big)
\Big(\beta_3\beta_4\alpha^{[2]}\Big),\\
&\Big(\beta_1\beta_2\beta_3\beta_4\Big)\Big(\alpha^{[4]}\Big)^3
\le \Big(\beta_1\alpha^{[3]}\Big)\Big(\beta_2\alpha^{[3]}\Big)
\Big(\beta_3\alpha^{[3]}\Big)\Big(\beta_4\alpha^{[3]}\Big). 
\intertext{The inequalities associated with indecomposable
$2$-covers of $B$ are }
&\Big(\beta_1\beta_2\beta_3\beta_4\Big)^2\Big(\alpha^{[4]}\Big)^1
\le \Big(\beta_1\beta_2\alpha^{[2]}\Big)
\Big(\beta_1\beta_3\beta_4\alpha^{[1]}\Big)\Big(\beta_2\beta_3\beta_4\alpha^{[1]}\Big),\\
&\Big(\beta_1\beta_2\beta_3\beta_4\Big)^2\Big(\alpha^{[4]}\Big)^2
\le \Big(\beta_1\alpha^{[3]}\Big)\Big(\beta_2\beta_3\alpha^{[2]}\Big)
\Big(\beta_2\beta_4\alpha^{[2]}\Big)\Big(\beta_1\beta_3\beta_4\alpha^{[1]}\Big),\\
&\Big(\beta_1\beta_2\beta_3\beta_4\Big)^2\Big(\alpha^{[4]}\Big)^3
\le \Big(\beta_1\alpha^{[3]}\Big)\Big(\beta_1\alpha^{[3]}\Big)
\Big(\beta_2\beta_3\alpha^{[2]}\Big)
\Big(\beta_2\beta_4\alpha^{[2]}\Big)\Big(\beta_3\beta_4\alpha^{[2]}\Big),
\intertext{and the inequalities associated with indecomposable
$3$-covers of $B$ are}
&\Big(\beta_1\beta_2\beta_3\beta_4\Big)^3\Big(\alpha^{[4]}\Big)^1 
\le \Big(\beta_1\beta_2\beta_3\alpha^{[1]}\Big)\Big(\beta_1\beta_2\beta_4\alpha^{[1]}\Big)
\Big(\beta_1\beta_3\beta_4\alpha^{[1]}\Big)\Big(\beta_2\beta_3\beta_4\alpha^{[1]}\Big),\\
&\Big(\beta_1\beta_2\beta_3\beta_4\Big)^3\Big(\alpha^{[4]}\Big)^2
\le \Big(\beta_1\beta_2\alpha^{[2]}\Big)\Big(\beta_1\beta_3\alpha^{[2]}\Big)
\Big(\beta_1\beta_4\alpha^{[2]}\Big)
\Big(\beta_2\beta_3\beta_4\alpha^{[1]}\Big)\Big(\beta_2\beta_3\beta_4\alpha^{[1]}\Big).
\end{align*}
There are no indecomposable $r$-covers of $B$ with $r \ge 4$ in this case. 
When $B$ is a set, Alon and Berman \cite{AlonBerman} gave an upper bound on $r$ for indecomposable $r$-covers of $B$ in terms of the cardinality of $B$, simplifying an earlier argument of Huckemann, Jurkat, and Shapley from game theory \cite{Graver}.

We study these inequalities through four classes of \emph{volume polynomials} reviewed in Section~\ref{sec:volume}:
projective volume polynomials, K\"ahler volume polynomials,  Hermitian volume polynomials, and convex body
volume polynomials. Their normalized coefficients are, respectively, intersection numbers of nef divisors, integrals of cup products of nef $(1,1)$-classes, mixed
discriminants of positive semidefinite matrices, and mixed volumes of convex bodies, up to taking limits.  
We compare the four classes of volume polynomials in
 Section~\ref{sec:relations}. In particular, we show  that  every Hermitian volume polynomial  is a projective volume polynomial and every convex body
volume polynomial is a
projective volume polynomial. These inclusions transfer inequalities for intersection numbers to those for mixed volumes and mixed discriminants.

We now express our main inequality in terms of projective volume polynomials. Hereafter, for a positive integer $n$, we write $\llbracket n \rrbracket$ for the set of indices $\{1,2,\ldots,n\}$. 
Let $X$ be a $d$-dimensional projective variety over an
algebraically closed field, and let $\beta=(\beta_1,\ldots,\beta_n)$
be a tuple of nef divisors on $X$. The \emph{volume polynomial} associated to $\beta$ is the homogeneous polynomial
\[
f(x)=f_\beta(x_1,\dots,x_n)\coloneq \frac{1}{d!}\deg_X( (x_1\beta_1+\cdots+x_n\beta_n)^d).
\]
For any multiset $B$ of cardinality $d$ on $\llbracket n \rrbracket$,  the coefficient of the monomial $x^B=x_1^{B(1)}\cdots x_n^{B(n)}$ in the projective volume polynomial associated to $\beta$ is the degree of the class $\beta^{[B]}$ on $X$.
Similarly, if $\delta=(\delta_1,\ldots,\delta_n)$ is a collection of nef $(1,1)$-classes on a $d$-dimensional compact K\"ahler manifold $Y$, the \emph{volume polynomial} associated to $\delta$ is the homogeneous polynomial
\[
f(x)=f_\delta(x_1,\dots,x_n)\coloneq \frac{1}{d!}\int_Y (x_1\delta_1+\cdots+x_n\delta_n)^d.
\]
For a multiset $S$ on $\llbracket n \rrbracket$, we consider the divided power differential operator
\[
\partial^{[S]} \coloneq \partial_1^{[S(1)]} \partial_2^{[S(2)]} \cdots \partial_n^{[S(n)]}, \ \ \text{where} \ \ \partial_i^{[k]}\coloneq \frac{1}{k!}\left(\frac{\partial}{\partial x_i}\right)^k.
\]
Taking $\alpha=\sum_{i=1}^n x_i\beta_i$ in
Theorem~\ref{thm introduction} for $x \in \mathbb{R}^n_{\ge 0}$, we obtain the following condition for $f(x)$.

\begin{corollary}\label{thm:intro-volume-polynomials}
Let $I=(I_1,\ldots,I_m)$ be an $r$-cover of a multiset $B$ of cardinality $d$ on $\llbracket n \rrbracket$.  For any volume polynomial $f(x)$ associated to an $n$-tuple of nef divisors on a $d$-dimensional projective variety, we have
\begin{equation*}\label{eq:intro-volume-polynomial-inequality}
\binom{B}{I}
(\partial^{[B]}f)^r f(x)^{m-r}
\leq 
\partial^{[I_1]}f(x) \
\partial^{[I_2]}f(x) \
 \cdots \
\partial^{[I_m]}f(x) \ \ \text{for all $x \in \mathbb{R}^n_{\ge 0}$.}
\end{equation*}
The same inequality holds for any volume polynomial $f(x)$ associated to an $n$-tuple of nef $(1,1)$-classes on a $d$-dimensional compact K\"ahler manifold.
\end{corollary}

For a nonnegative vector $u=(u_1,\ldots,u_n)$, we write $\partial_u$ for the corresponding directional derivative $\sum_{i=1}^n u_i  \partial_i$. A homogeneous polynomial $f$ of degree $d \ge 2$ in $n$ variables with nonnegative coefficients is \emph{Lorentzian} if, for all  $v_1,\ldots,v_d \in \mathbb{R}^n_{\ge 0}$, we have
\[
 \big( \partial_{v_1}  \partial_{v_1} \partial_{v_3}  \cdots \partial_{v_d} f \big)  \big( \partial_{v_2}  \partial_{v_2} \partial_{v_3}  \cdots \partial_{v_d} f\big)   \le \big( \partial_{v_1} \partial_{v_2} \partial_{v_3}\cdots \partial_{v_d} f \big)^2.
\]
According to \cite[Theorem 4.6]{branden2020lorentzian}, every projective volume polynomial is a Lorentzian polynomial. 
Corollary~\ref{thm:intro-volume-polynomials} provides a way to show that a given Lorentzian polynomial is not a projective volume polynomial, and hence not a convex body volume polynomial; see Examples~\ref{ex:rKT violation} and~\ref{ex:new inequality}. Such examples show that Theorem~\ref{thm introduction} cannot be formally deduced from the Lorentzian property of volume polynomials.

The theory of volume and covolume polynomials developed in
\cite{AluffiCovolume,LinearOperator} provides further extensions of
Theorem~\ref{thm introduction}, as differential operators associated
with projective covolume polynomials preserve projective volume
polynomials; see \cite[Section 4]{HuhVolume} for a survey.
For example, we may insert a common product of auxiliary nef divisor
classes into each intersection product in Theorem~\ref{thm introduction},
as follows.

\begin{corollary}\label{cor:monomial}
Let $B$ be a multiset of cardinality $d$ on $\llbracket n \rrbracket$, and let $I=(I_1,\ldots,I_m)$ be an $r$-cover of $B$. 
For any nef divisors $\alpha,\beta_1,\ldots,\beta_n,\gamma_1,\ldots,\gamma_k$ on a projective variety of dimension $d+k$,  
\[
\binom{B}{I}
\left(\beta^{[B]} \ \gamma \right)^r
\left(\alpha^{[d]} \ \gamma  \right)^{m-r}
\leq 
\left(\beta^{[I_1]} \ \alpha^{[d-|I_1|]} \ \gamma  \right)
\cdots
\left(\beta^{[I_m]} \ \alpha^{[d-|I_m|]} \ \gamma  \right),
\]
where $\beta=(\beta_1,\ldots,\beta_n)$ and $\gamma=\gamma_1 \cdots \gamma_k$.
\end{corollary}

The same inequality holds more generally when $\gamma$ is the evaluation of any degree $k$ projective covolume polynomial at nef divisor classes on the projective variety. We do not know if the analogue of Corollary~\ref{cor:monomial} holds for nef $(1,1)$-classes on compact K\"ahler manifolds.

Since  every Hermitian volume polynomial is a
projective volume polynomial and every convex body
volume polynomial is a
projective volume polynomial,
the inequality of Corollary~\ref{cor:monomial} also holds when $\alpha,\beta_i,\gamma_j$ are positive semidefinite Hermitian matrices of size $d+k$, with mixed discriminants replacing intersection numbers, and when $\alpha,\beta_i,\gamma_j$ are  convex bodies in $\RR^{d+k}$, with mixed volumes replacing intersection numbers. 
In both settings, products specify lists of arguments, and divided powers indicate repeated arguments with the  factorial normalization; see Section~\ref{sec:definitions} for mixed discriminants and mixed volumes, and Remark~\ref{sec_divided power} for our notational convention.

\begin{corollary}\label{cor:discriminant main}
Let $B$ be a multiset of cardinality $d$ on $\llbracket n\rrbracket$, and let $I=(I_1,\ldots,I_m)$ be an $r$-cover of $B$. 
For any positive semidefinite matrices $A,M_1,\ldots,M_n,N_1,\ldots,N_k$ of size $d+k$, we have 
\[
\binom{B}{I}
\D\left[M^{[B]}, N \right]^r
\D\left[A^{[d]}, N  \right]^{m-r}
\leq 
\D\left[M^{[I_1]}, A^{[d-|I_1|]}, N \right]
\cdots
\D\left[M^{[I_m]}, A^{[d-|I_m|]}, N \right],
\]
where $\D$ is the mixed discriminant, $M=(M_1,\ldots,M_n)$, and $N=(N_1,\ldots,N_k)$.
\end{corollary}

For example, consider the case $k=0$. When $I$ is the partition of $\llbracket n \rrbracket$ into singletons and $A$ is the identity matrix, Corollary~\ref{cor:discriminant main} states that, for any  positive semidefinite matrices $M_1,\ldots,M_d$ of size $d$, we have
\[
\D[M_1,\ldots,M_d] \le \frac{1}{d!}\tr(M_1)\cdots\tr(M_d).
\]
 Equality holds when each $M_i$ is a rank-one diagonal matrix whose $i$-th diagonal entry is $1$. As another example, fix these matrices $M_i$ and allow $A$ to vary. When $I$ is a bipartition of $\llbracket n \rrbracket$, Corollary~\ref{cor:discriminant main} for $k=0$ becomes Fischer's inequality for $A$.

\begin{corollary}\label{cor: mixed volumes}
Let $B$ be a multiset of cardinality $d$ on $\llbracket n \rrbracket$, and let $I=(I_1,\ldots,I_m)$ be an $r$-cover of $B$. 
For any convex bodies $K,P_1,\ldots,P_n,Q_1,\ldots,Q_k$   in $\mathbb{R}^{d+k}$, we have
\[
\binom{B}{I}
\MV\left(P^{[B]}, Q \right)^r
\MV\left(K^{[d]}, Q \right)^{m-r}
\leq 
\MV\left(P^{[I_1]}, K^{[d-|I_1|]}, Q \right)
\cdots
\MV\left(P^{[I_m]}, K^{[d-|I_m|]}, Q \right),
\]
where $\MV$ is the mixed volume, $P=(P_1,\ldots,P_n)$, and $Q=(Q_1,\ldots,Q_k)$. 
\end{corollary}

The Alexandrov--Fenchel inequality is the case $d=2$ and $n=1$. 
It would be interesting to understand the equality cases of  Corollaries~\ref{cor:discriminant main} and~\ref{cor: mixed volumes}. For characterizations of equality in special cases, see \cite{Panov} for mixed discriminants and \cite{shenfeld2023extremals} for mixed volumes.

\begin{example}
\
\begin{enumerate}[(1)]\itemsep 2pt
\item Let $E$ be a set of cardinality $1$, and consider the $1$-cover $I=(E,E)$ of $B=2E$. In this case, Corollary~\ref{cor: mixed volumes} specializes to the Alexandrov--Fenchel inequality
\[
\MV(P,P,Q_1,\ldots,Q_k)\MV(K,K,Q_1,\ldots,Q_k)
\leq \MV(P,K,Q_1,\ldots,Q_k)^2,
\]
where $P,K,Q_1,\ldots,Q_k$ are convex bodies in $\mathbb{R}^{k+2}$ and $\MV$ denotes their mixed volume. Note that the constant $\binom{B}{I}=\binom{2}{1} \cdot \binom{2}{1}$ cancels the two factors of $2!$ arising from the divided powers. 
\item Let $E$ be a set of cardinality $4$, and consider a tripartition $E=I_1 \sqcup I_2 \sqcup I_3$. In this case, up to relabeling, Corollary~\ref{cor: mixed volumes} specializes to the inequality
\begin{multline*}
\MV(P_1,P_2,P_3,P_4,Q_1,\ldots,Q_k) \MV(K,K,K,K,Q_1,\ldots,Q_k) \MV(K,K,K,K,Q_1,\ldots,Q_k) \le \\
8  \MV(P_1,K,K,K,Q_1,\ldots,Q_k) \MV(P_2,K,K,K,Q_1,\ldots,Q_k) \MV(P_3,P_4,K,K,Q_1,\ldots,Q_k),
\end{multline*}
where $P_1,P_2,P_3,P_4,K,Q_1,\ldots,Q_k$ are convex bodies in $\mathbb{R}^{k+4}$. 
Equality holds e.g.~when
\[
K=\mathscr{E}_{1}+\mathscr{E}_{2}+\mathscr{E}_{3}+\mathscr{E}_{4}, \ 
P_1=\mathscr{E}_{1}, \ P_2=\mathscr{E}_{2}, \ P_3=P_4=\mathscr{E}_{3}+\mathscr{E}_{4}, \ Q_1=\mathscr{E}_5,\dots,Q_k=\mathscr{E}_{k+4},
\] 
where $\mathscr{E}_i$ is the unit segment in the $i$-th coordinate direction. 
Example~\ref{ex:new inequality} shows that this inequality cannot be formally deduced from the Alexandrov--Fenchel inequality. 
\item Consider the case $k=0$ and $B=\llbracket n \rrbracket$. Taking $P_i=\mathscr{E}_i$, the left-hand side of Corollary~\ref{cor: mixed volumes} is
\[
\MV\left(P^{[B]} \right)^r
\MV\left(K^{[d]} \right)^{m-r}=  \left(\frac{1}{d!}\right)^{r} \cdot \left(\frac{1}{d!}\right)^{m-r} \vol(K)^{m-r}.
\]
On the other hand, 
by the mixed-volume formula for orthogonal projections \cite[Section 5.3]{Schneider:BM}, we have
\[
\MV\left(P^{[I_1]}, K^{[d-|I_1|]} \right) \cdots \MV\left(P^{[I_m]}, K^{[d-|I_m|]} \right) =\frac{1}{d!}\vol(K_{\widehat{I_1}}) \cdots \frac{1}{d!}\vol(K_{\widehat{I_m}}),
\]
where the volume of $K_{\widehat{I_j}}$ is taken in its ambient space $\mathbb{R}^{B \setminus I_j}$.
Thus, in this case, Corollary~\ref{cor: mixed volumes}  becomes the Bollob\'as--Thomason inequality for the convex body $K$. 
\end{enumerate}
\end{example}

\begin{example}\label{ex:optimal}
We note that 
 the constant $\binom{B}{I}$ is best possible  in Theorem~\ref{thm introduction} and Corollary~\ref{cor:monomial}. 
To see this, consider the product of projective spaces
\[
X=\left(\prod_{e\in E} \PP^{B(e)}\right) \times \mathbb{P}^k.
\]
Let $\beta_e$ be the pullback of the hyperplane class from the
$e$-th factor $\PP^{B(e)}$, let
$\alpha$ be the sum of all $\beta_e$, and  let $\gamma$ be the $k$-fold product of the pullback of the hyperplane class from $\mathbb{P}^k$. 
In this case, both sides of the inequality in  Corollary~\ref{cor:monomial} are positive and equal to $\binom{B}{I}/ (B!)^m$.
Similar examples show that the constant $\binom{B}{I}$ cannot be improved  in Corollaries~\ref{cor:discriminant main} and~\ref{cor: mixed volumes}.
\end{example}

\subsection*{Outline of proof of Theorem~\ref{thm introduction}}

We begin with the compact K\"ahler case, which makes
the relation between mixed discriminants of positive semidefinite matrices and integrals of nef $(1,1)$-classes 
explicit.

\begin{enumerate}[(1)]\itemsep 2pt

\item Yau's solution of the Calabi conjecture
\cite[Theorem~1]{yau1977calabi} allows us to prescribe the volume
form of a K\"ahler metric within a fixed K\"ahler class, subject
to the total-volume constraint. Together with H\"older's
inequality, this allows us to transfer the analogue of Theorem~\ref{thm introduction} for positive semidefinite matrices to   the corresponding inequality for nef $(1,1)$-classes (Proposition~\ref{thm:disc=>Kahler}).

\item  By an elementary linear algebra argument, we prove  the analogue of Theorem~\ref{thm introduction} for positive semidefinite matrices when $B$ is a set, which is Corollary~\ref{cor:discriminant main} when
$B$ is a set and $k=0$ (Proposition~\ref{prop:disc-exponents-1}).
By step (1), we obtain
Theorem~\ref{thm introduction} for nef $(1,1)$-classes on compact
K\"ahler manifolds when $B$ is a set. In particular, we obtain
Theorem~\ref{thm introduction} for nef divisors on complex projective varieties when $B$ is a set (Corollary~\ref{cor:exponents1intersection}).

\item Upper truncation and polarization operators preserve projective
volume polynomials; see \cite[Corollary 3.3 and Proposition 4.1]{LinearOperator}. Using these operators
and the arithmetic--geometric mean inequality, we pass from
covers of sets to covers of multisets. This proves
Theorem~\ref{thm introduction}  for nef divisors on complex projective varieties for any multiset $B$. This step has no known analogue for nef $(1,1)$-classes on compact K\"ahler manifolds.

\item  Since  every Hermitian volume polynomial is a
projective volume polynomial,
we may transfer Theorem~\ref{thm introduction}  and Corollary~\ref{cor:monomial} for nef divisors on
complex projective varieties to the corresponding inequality
for positive semidefinite matrices.
This proves Corollary~\ref{cor:discriminant main}.

\item Finally, we apply the transfer technique in step (1) to Corollary~\ref{cor:discriminant main} when $k=0$.
This proves  Theorem~\ref{thm introduction} for nef $(1,1)$-classes on compact
K\"ahler manifolds for any multiset $B$. 
\end{enumerate}

For the algebraic case, we study the volume function on the real
N\'eron--Severi space $N^1(X)_\RR$ of a projective variety $X$ over an algebraically closed field.
Using the theory of restricted volumes in 
\cite{elmnp09restrictedvolume} and \cite{lazamustNewtonOkou}, we prove
Theorem~\ref{thm introduction} directly when $B$ is a set in Section~\ref{sec:volproj}.
This replaces steps~(1) and (2) with
a birational-geometric argument. Step (3) remains valid over any algebraically closed field; see Section~\ref{sec:setmultiset}. 

In Section~\ref{sec:volume}, we review and develop basic properties of various notions of volume polynomials. 
In particular, we show that
every projective volume polynomial over $\CC$ is a K\"ahler volume polynomial (Proposition~\ref{cor:projective are Kahler}),
 every convex body volume polynomial is a projective volume polynomial over any algebraically closed field (Proposition~\ref{prop:convex bodies are projective}), and 
every Hermitian volume polynomial is a projective volume polynomial over any algebraically closed field (Proposition~\ref{hermitian vol}). 
In Section~\ref{sec:mixed disc inequalities}, we prove the transfer theorem (Proposition~\ref{thm:disc=>Kahler}) and 
the analogue of Theorem~\ref{thm introduction} for positive semidefinite matrices when
$B$ is a set (Proposition~\ref{prop:disc-exponents-1}).
Sections~\ref{sec:volproj} and ~\ref{sec:setmultiset} are devoted to the proof of Theorem~\ref{thm introduction} for nef divisors on projective varieties. In Section~\ref{sec:applications}, we collect various special cases of Theorem~\ref{thm introduction}.

\subsection*{Acknowledgements}
June Huh was partially supported by the Simons Investigator Grant
MPS-SIM-00825156 and NSF grant DMS-2053308.  Mateusz Micha{\l}ek was partially supported by the Charles Simonyi Endowment at the Institute for Advanced Study and the DFG grant 580118961.
Botong Wang was partially supported by the NSF grant DMS-1926686 and the Simons Foundation. Jian Xiao was supported by the National Key Research and Development Program of China 2021YFA1002300.
Part of this work was carried out during the 2024--2025 special year on Algebraic and Geometric Combinatorics at the Institute for Advanced Study. The authors thank Sergio Cristancho, Tuong Le, Chayim Lowen, Shouda Wang, and other participants of the special year for helpful discussions during the preparation of this manuscript.

\subsection*{AI disclosure}

LLMs were used for proofreading and editing. ChatGPT-5.6 assisted us in finding Example~\ref{ex:new inequality}.

\section{Different notions of volume polynomials}\label{sec:volume}

\subsection{Definitions}\label{sec:definitions}
Throughout the paper, a variety is a reduced, irreducible, separated scheme of finite type over an algebraically closed field. All compact K\"ahler manifolds are assumed to be connected. 
Limits of polynomials are taken coefficientwise in fixed degree and a fixed number of variables.

In \cite{Minkowski}, Minkowski made a foundational observation on Minkowski sums of convex bodies that has since become a cornerstone of convex geometry:
For any collection of convex bodies $P=(P_1,\ldots,P_n)$  in $\mathbb{R}^d$, the function
\[
f_P:\mathbb{R}^n_{\ge 0} \longrightarrow \mathbb{R}_{\ge 0}, \qquad (x_1,\ldots,x_n) \longmapsto \frac{1}{d!} \vol(x_1P_1+\cdots+x_nP_n)
\]
is a degree $d$ homogeneous polynomial in $x=(x_1,\ldots,x_n)$.
The polynomial $f_P(x)$, called the \emph{volume polynomial} associated to $P$, is then used to define the \emph{mixed volumes} of convex bodies in $P$ as its normalized coefficients
\[
\MV(P_{i_1},\ldots,P_{i_d}) \coloneq \frac{\partial}{\partial x_{i_1}} \cdots \frac{\partial}{\partial x_{i_d}} f_P(x_1,\ldots,x_n).
\]
The mixed volume $\MV$ is a symmetric multilinear function on $d$-tuples of convex bodies in $\mathbb{R}^d$. 
For a comprehensive introduction to the theory of mixed volumes, see \cite[Section 5.1]{Schneider:BM}. 
A \emph{realizable convex body volume polynomial} is a nonnegative rational multiple of a polynomial of the form $f_P(x)$, where each convex body in $P$ is a rational convex polytope. A homogeneous polynomial with real coefficients is called a \emph{convex body volume polynomial} if it is a limit of realizable convex body volume polynomials.

\begin{proposition}
A degree $d$ homogeneous polynomial  is a convex body volume polynomial if and only if it is of the form $f_P(x)$ for some collection of convex bodies $P$ in $\mathbb{R}^d$.
\end{proposition}

\begin{proof}
This is observed in \cite[Remark~4.3]{branden2020lorentzian}.
Since every convex body is the limit of rational convex polytopes, the polynomial $f_P(x)$ for any collection of convex bodies $P$ is a convex body volume polynomial. Conversely, by the compactness theorem of Shephard \cite[Theorem~1]{Shephard},  every convex body volume polynomial is of the above form for some collection of convex bodies.
\end{proof}

Let $M=(M_1,\ldots,M_n)$ be a collection of $d \times d$ positive semidefinite Hermitian matrices. The \emph{volume polynomial} associated to $M$ is the homogeneous polynomial
\[
f_M(x)=\frac{1}{d!}\det(x_1M_1+\cdots+x_nM_n).
\]
The \emph{mixed discriminants} of matrices in $M$ are defined as its normalized coefficients
\[
\D[M_{i_1},\dots,M_{i_d}]\coloneq \frac{\partial}{\partial x_{i_1}}\cdots\frac{\partial}{\partial x_{i_d}}f_M(x_1,\ldots,x_n).
\]
The mixed discriminant $\D$ is a symmetric multilinear function on $d$-tuples of $d \times d$ Hermitian matrices. 
For basic properties of mixed discriminants, we refer to \cite[Section 5.5]{Schneider:BM}. 
A \emph{realizable Hermitian volume polynomial} is a nonnegative rational multiple of a polynomial of the form $f_M(x)$, where every entry of every matrix in $M$ has rational real and imaginary parts. A homogeneous polynomial with real coefficients is called a \emph{Hermitian volume polynomial} if it is a limit of realizable Hermitian volume polynomials.

\begin{proposition}
 A degree $d$ homogeneous polynomial  is a Hermitian volume polynomial if and only if it is of the form $f_M(x)$ for some collection of $d \times d$ positive semidefinite matrices $M$.
\end{proposition}

\begin{proof}
Let $\mathcal H$ be the set of polynomials of the form $f_M(x)$, where $M=(M_1,\ldots,M_n)$ is a collection of $d \times d$ positive semidefinite matrices. It is enough to show that $\mathcal{H}$ is closed.

\emph{Step 1}: We prove that if $S \coloneq M_1+\dots+M_n$ is degenerate, then
\[
\det(x_1M_1+\cdots+x_nM_n)=0.
\]
Indeed, in this case, there is a nonzero vector $v$ in its kernel, so that 
\[
v^*(M_1+\cdots+M_n)v=0.
\]
Since each $M_i$ is positive semidefinite, we must have $v^*M_iv=0$ for every $i$. It follows that $M_iv=0$ for every $i$, and hence $x_1M_1+\cdots+x_nM_n$ is degenerate for any $x \in \mathbb{R}^n$.

\emph{Step 2}: If $f_M(x)$ is nonzero, the previous step shows that $S$ is positive definite. Thus, in this case, $S$
 has a well-defined positive definite square root $S^{1/2}$. By the multiplicativity of determinants, we have 
\[
f_M(x)=\frac{1}{d!}\det\left(\sum_{i=1}^n x_iM_i\right)=\frac{1}{d!}(\det S)\det\left(\sum_{i=1}^n x_iN_i\right), \ \ \text{where $N_i=S^{-1/2}M_iS^{-1/2}$,}
\]
and the matrices $N_i$ remain positive semidefinite. Thus, either $f_M(x)$ is zero or it is a positive multiple of $f_N(x)$, where $N=(N_1,\ldots,N_n)$ is a collection of $d \times d$ positive semidefinite matrices whose sum is the identity matrix.

\emph{Step 3}: Let $\mathcal{K}$ be the set of $n$-tuples of $d\times d$ Hermitian matrices
\[
\left\{N=(N_1,\dots,N_n)\;\middle\vert\;\sum_{i=1}^n N_i=\Id,\ \text{$N_i$ is positive semidefinite for every $1\leq i\leq n$}\right\}.
\]
This set is closed and bounded, and hence compact. The function $f$ which maps $N$ to $f_N(x)$ is continuous, and therefore the image $f(\mathcal{K})$ of $\mathcal{K}$ is compact, and hence closed. Further, $f(\mathcal{K})$ does not contain zero. Since $\mathcal{H}$ is the set of all nonnegative multiples of polynomials in $f(\mathcal{K})$, it is closed as well.
\end{proof}

Let $\beta=(\beta_1,\ldots,\beta_n)$ be a collection of nef divisors on a $d$-dimensional projective variety $X$ over an algebraically closed field $\mathbb{K}$. The \emph{volume polynomial} associated to $\beta$ is
\[
f_\beta(x)=\frac{1}{d!}\deg_X (x_1\beta_1 +\cdots + x_n\beta_n)^d.
\]
A \emph{realizable projective volume polynomial over $\mathbb{K}$} is a nonnegative rational multiple of a polynomial of the form $f_\beta(x)$, where each divisor in $\beta$ is semiample.\footnote{A $\mathbb{Q}$-divisor is \emph{semiample} if some positive integer multiple of it is a base-point free divisor \cite[Section 2.1]{lazarsfeldPosI}.} A homogeneous polynomial with real coefficients is called a \emph{projective volume polynomial over $\mathbb{K}$} if it is a limit of realizable projective volume polynomials over $\mathbb{K}$.
Similarly, if $\delta=(\delta_1,\ldots,\delta_n)$ is a collection of nef $(1,1)$-classes on a $d$-dimensional compact K\"ahler manifold $Y$, the \emph{volume polynomial} associated to $\delta$ is the homogeneous polynomial
\[
f_\delta(x)\coloneq \frac{1}{d!}\int_Y (x_1\delta_1+\cdots+x_n\delta_n)^d.
\]
A \emph{K\"ahler volume polynomial} is a limit of polynomials of the form $f_\delta(x)$. 

\begin{remark}
A number of basic questions on various volume polynomials remain  unanswered. We record some of them here.
\begin{enumerate}[(1)]\itemsep 2pt
\item Is every projective volume polynomial of the form $f_\beta(x)$ for some collection of nef real divisors $\beta$ on a projective variety? 
\item Is every K\"ahler volume polynomial of the form $f_\delta(x)$ for some collection of nef $(1,1)$-classes $\delta$ on a compact K\"ahler manifold?
\item Is the set of projective volume polynomials over $\mathbb{C}$ properly contained in the set of K\"ahler volume polynomials?
\end{enumerate}
\end{remark}

\subsection{Relations}\label{sec:relations}

We compare the sets of different types of volume polynomials.

\begin{proposition}\label{cor:projective are Kahler}
Every projective volume polynomial over $\CC$ is a K\"ahler volume polynomial.
\end{proposition}

\begin{proof}
Let $f_\beta(x)$ be the volume polynomial associated to a collection of semiample divisors $\beta$ on a complex projective variety $X$. Using Hironaka's resolution of singularities \cite{Hironaka}, we may find a projective birational map $Y \to X$ with $Y$ smooth. Since $Y$ is a compact K\"ahler manifold and the pullback of $\beta$ consists of nef $(1,1)$-classes, we have the statement.
\end{proof}

\begin{proposition}\label{prop:convex bodies are projective}
 Every convex body volume polynomial is a projective volume polynomial over any algebraically closed field.
\end{proposition}

\begin{proof}
It is enough to show that every realizable convex body volume polynomial is a realizable projective volume polynomial over $\K$. 
Let $f_P(x)$ be the volume polynomial associated to a collection of rational polytopes $P$. 
Scaling by positive integers, we may assume that each polytope in $P$ is a lattice polytope.  In this case, it follows from the construction of toric varieties that $f_P(x)$ is a realizable projective volume polynomial over $\K$; see  \cite[Section~5.4]{fulton1993introduction}. 
\end{proof}

We now show that every Hermitian volume polynomial is a projective volume polynomial. The argument works more generally for positive semidefinite matrices over the noncommutative algebra of quaternions $\mathbb{H}$; see \cite[Chapter 1]{Voight} for background on quaternions. 
As in the complex case, we define an $\mathbb H$-Hermitian matrix $M$ to be a square matrix with quaternionic entries such that $M$ is equal to its conjugate transpose. 
For a quaternionic matrix, 
we will always refer to right eigenvalues as the eigenvalues. 
An $\mathbb{H}$-Hermitian matrix is positive definite if and only if its eigenvalues are positive, and 
 an $\mathbb{H}$-Hermitian matrix is positive semidefinite if and only if its eigenvalues are nonnegative \cite[Remark~6.1]{zhang1997quaternions}. When a distinction is necessary, we refer to the usual Hermitian matrices as $\mathbb{C}$-Hermitian matrices.

Due to the lack of commutativity, there are several concepts of determinants for a quaternionic matrix $M$. 
We will use the Moore determinant introduced in \cite{dresden1922twenty}, denoted  $\det M$. We recall the definition as in \cite{MooreDeterminantFormula}.
\begin{definition}
Let $M=(m_{ij})$ be an $\mathbb{H}$-Hermitian matrix of size $d$. Any permutation $\sigma\in S_d$ can be uniquely written as a composition of disjoint cycles
\[
\sigma=(i_{11}\cdots i_{1q_1})(i_{21}\cdots i_{2q_2})\cdots(i_{t1}\cdots i_{tq_t}),
\]
where $i_{a1}<i_{ab}$ for all $b>1$ and $i_{11}>i_{21}>\cdots>i_{t1}$.
We define the \emph{Moore determinant} by
\[
\det(M)=\sum_{\sigma\in S_d}\sgn(\sigma)m_{i_{11}i_{12}}\cdots m_{i_{1q_1}i_{11}}m_{i_{21}i_{22}}\cdots m_{i_{tq_t}i_{t1}},
\]
where $\sgn(\sigma)$ is the sign of the permutation and a one-cycle contributes its diagonal entry.
\end{definition}

The Moore determinant of an $\mathbb{H}$-Hermitian matrix $M$ satisfies the following properties; see \cite[Section 1]{alesker2003non} for an overview.
\begin{enumerate}[(1)]\itemsep 2pt
\item $\det(M)$ is a real number.
\item $\det(M)$ is the product of eigenvalues of $M$.
\item $\det(M)$ vanishes if and only if there exists a nonzero vector in the kernel of $M$.
\item $\det(M)$ coincides with the usual determinant when the entries of $M$ are complex numbers.
\end{enumerate}
Using these properties one may extend the notion of mixed discriminant to the quaternionic case \cite[Definition~1.1.14]{alesker2003non}. Relatedly, for  a collection of $d \times d$ positive semidefinite $\mathbb{H}$-Hermitian matrices $M=(M_1,\ldots,M_n)$, we may define the \emph{volume polynomial} associated to $M$ to be the homogeneous polynomial
\[
f_M(x)=\frac{1}{d!}\det(x_1M_1+\cdots+x_nM_n).
\]
We say that $f_M(x)$ is a \emph{realizable $\mathbb{H}$-Hermitian volume polynomial} when every entry of every matrix in $M$ has rational real and $i$, $j$, $k$ parts. A homogeneous polynomial with real coefficients is called an \emph{$\mathbb{H}$-Hermitian volume polynomial} if it is a limit of realizable $\mathbb{H}$-Hermitian volume polynomials.

In \cite[Corollary 1.1.16]{alesker2003non}, Alesker proved that analogues of the Alexandrov--Fenchel inequality hold for mixed discriminants of positive semidefinite $\mathbb{H}$-Hermitian matrices.  This result implies that every $\mathbb{H}$-Hermitian volume polynomial is a  Lorentzian polynomial. The following proposition shows that $\mathbb{H}$-Hermitian volume polynomials are, in fact, projective volume polynomials.

\begin{proposition}\label{hermitian vol}\hfill
\begin{enumerate}[(1)]\itemsep 2pt
\item Every $\mathbb{C}$-Hermitian volume polynomial is a projective volume polynomial over any algebraically closed field.
\item Every $\mathbb{H}$-Hermitian volume polynomial is a projective volume polynomial over any algebraically closed field of positive characteristic.
\end{enumerate}
\end{proposition}
\begin{proof}
The proofs of the two statements are essentially identical. We choose an elliptic curve $E$ with nontrivial endomorphisms;  for the existence, see \cite[Chapter V]{Silverman}. For (1), we choose an elliptic curve $E$ such that $\End(E)_\QQ$ is an imaginary quadratic field. 
For (2), we choose an elliptic curve $E$ such that $\End(E)_\RR \simeq\mathbb{H}$. We will write $\mathbb{F}$ for $\CC$ in the first case and for $\mathbb{H}$ in the second case.
Let $X$ be the $d$-fold product $E^d$. We have an identification of the endomorphism algebra of the abelian variety $X$, tensored with $\RR$, with the ring of $d\times d$ matrices with entries in $\mathbb{F}$. By \cite[Section 20]{Mumford1970AbelianVarieties}, we have an embedding of the N\'eron-Severi space $\NS_\QQ(X)$ inside $\End(X)_\QQ$, where the image is the locus of the invariants of the Rosati involution. Using the product polarization and identifying the real endomorphism algebra with matrices, the Rosati involution is identified with conjugate transpose; see \cite[Corollary~3.5]{prendergast2012cone}. Thus, we have a linear isomorphism $\psi$ between $\NS_\RR(X)$ and the space of $\mathbb{F}$-Hermitian matrices. Further, by \cite[Theorem~4.3]{prendergast2012cone}, this identification maps the nef cone exactly to the cone of positive semidefinite matrices.

We show that there is a positive constant $c$ such that, for any $\alpha \in\NS_\RR(X)$, we have 
\[
\deg_X  (\alpha^d)=c\cdot\det\psi(\alpha).
\]
Note that the left-hand side and the right-hand side are both homogeneous polynomials of degree $d$ on the real vector space $\NS_\RR(X)$. As the nef cone coincides with the pseudoeffective cone for abelian varieties \cite[Example 1.4.7]{lazarsfeldPosI}, both functions vanish on the boundary of the nef cone.

We claim that $\det=0$ defines an irreducible hypersurface of degree $d$ in the vector space
\[
\Big\{\text{$\mathbb{F}$-Hermitian matrices of size $d\times d$}\Big\}\otimes_\RR\CC.
\]
Since the Moore determinant is a homogeneous polynomial of degree $d$, it defines a hypersurface of degree $d$ in the above complex vector space. Moreover, the claim is evident after restricting to the subspace of real symmetric matrices, where the determinant polynomial is irreducible. The claim follows, since its restriction to a linear subspace is irreducible and has the same degree. 

If the two polynomials were not proportional, then, by the irreducibility established above, their common zero locus in the complex vector space $\NS_\CC(X)$ would have codimension at least two. Let $Z$ be the Zariski closure of the boundary of the nef cone in $\NS_\RR(X)$. Since the boundary of the cone of positive semidefinite matrices has real codimension one and is contained in the determinant hypersurface, $Z$ has codimension one, as does its complexification $Z_\CC$. On the other hand, $Z_\CC$ is contained in the common zero locus of the two polynomials, a contradiction. Thus, the polynomials are proportional, with a positive proportionality constant $c$ since both are positive on the ample cone.

The claim allows us to find, for any collection of positive semidefinite matrices $M$,  a corresponding collection of nef real divisor classes $\beta$ such that $f_\beta(x)$ is proportional to $f_M(x)$. Approximating $\beta$ by rational ample classes, we get  the statement of the proposition.
\end{proof}

\begin{remark}\label{rmk:volume polynomials}
Some further properties of the various volume polynomials are listed below.
\begin{enumerate}[(1)]\itemsep 2pt
\item There are projective volume polynomials over any algebraically closed field that are not convex body volume polynomials. 
The distinction is already visible in degree two and four variables;
see \cite[Theorem~1.6]{huang2025realizations} and \cite[Theorem 3.2]{Averkov-Soprunov} for characterizations
of the corresponding off-diagonal coefficient vectors.
In particular, the set of projective volume polynomials over any algebraically closed field is strictly larger than the set of convex body volume polynomials.
\item\label{item2} A polynomial with real coefficients is called \emph{stable} if it is identically zero or does not vanish at any point in the product of open upper half-planes. All nonzero Hermitian volume polynomials are stable \cite[Proposition~2.4]{borcea2008applications}. In fact, the same holds for $\mathbb{H}$-Hermitian volume polynomials; see \cite[Example~2.8]{Gur06}.
\item All projective volume polynomials and K\"ahler volume polynomials are Lorentzian. For projective volume polynomials, this is proved in \cite[Theorem~4.6]{branden2020lorentzian}. Using Gromov's version of the Hodge index theorem \cite{gromov1990convex},  
one can similarly deduce that all K\"ahler volume polynomials are Lorentzian.
\item\label{item4} 
Over any algebraically closed field $\K$, there exist projective volume polynomials over $\mathbb{K}$ that are not stable. 
In fact, every bivariate Lorentzian polynomial is a convex body
volume polynomial by \cite[Theorem~4]{Shephard}, and hence is a projective
volume polynomial over any algebraically closed field by
Proposition~\ref{prop:convex bodies are projective}.
  For example, the polynomial $2x^3+12x^2y+18xy^2+\frac{17}{2}y^3$ is a projective volume polynomial that is not stable \cite[Example~2.3]{branden2020lorentzian}.
\item Combining (\ref{item2}) and (\ref{item4}), we see that there exist projective volume polynomials that are not Hermitian volume polynomials.
\item Every product of linear forms with nonnegative coefficients is a volume polynomial of each kind. In fact, using hyperrectangles and diagonal matrices, we can realize such polynomials as convex body volume polynomials and Hermitian volume polynomials. 
\end{enumerate}
\end{remark}

The following corollary is an immediate consequence of Propositions \ref{prop:convex bodies are projective} and \ref{hermitian vol}.

\begin{corollary}\label{cor:projective strong}
Any non-strict inequality of intersection numbers of nef real divisor classes on smooth projective varieties also applies to mixed volumes of convex bodies and mixed discriminants of positive semidefinite matrices.
\end{corollary}

For example, consider the following special case of
Corollary~\ref{cor:monomial}: For any nef divisors $\alpha$, $\beta_1,\beta_2,\beta_3$, $\gamma$ on a $7$-dimensional projective variety, we have
\[
\frac25 \ (\beta_1^2\beta_2^2\beta_3^2 \gamma)\ (\alpha^6  \gamma) \leq (\beta_1\beta_2\beta_3\alpha^3  \gamma)\ (\beta_1\beta_2\beta_3\alpha^3  \gamma).
\]
Thus, by Corollary~\ref{cor:projective strong},  for any $7 \times 7$ positive semidefinite matrices $A,M_1,M_2,M_3,N$, 
\[
\frac25\D\left[M_1^{\times2},M_2^{\times2},M_3^{\times2},N\right]\D\left[A^{\times6},N\right] \leq \D\left[M_1,M_2,M_3,A^{\times3},N\right]
\D\left[M_1,M_2,M_3,A^{\times3},N\right].
\]
Similarly, for any convex bodies $K,P_1,P_2,P_3,Q\subseteq\RR^7$, we have
\[
\frac25\MV\left(P_1^{\times2},P_2^{\times2},P_3^{\times2},Q\right)\MV\left(K^{\times6},Q\right) \leq \MV\left(P_1,P_2,P_3,K^{\times3},Q\right) \MV\left(P_1,P_2,P_3,K^{\times3},Q\right).
\]
It would be interesting to find a convex-geometric proof of such mixed volume inequalities.

\begin{remark}\label{sec_divided power}
If a convex body $P_i$ appears multiple times in a mixed volume, we use the notation $P_i^{\times k}$ to denote $P_i$ appearing $k$ times.
We use the same notation to denote a matrix appearing multiple times in a mixed discriminant as well. 
For a convex body $P_i$, we use $P_i^{[k]}$ to denote the divided power $P_i^{\times k}/k!$. More generally, for a multiset $B$ of cardinality $d$ on $\llbracket n \rrbracket$ and a collection $P=(P_1,\dots,P_n)$ of convex bodies in $\RR^d$, we set
\[
\MV\left(P^{[B]}\right)\coloneq \frac{1}{B!}\MV\left(P^{ B}\right),
\]
where $P^{ B}$ is the list consisting of $B(i)$ copies of $P_i$ for each $i$ and $B!$ is the product of $B(i)!$ for all $i$. We also use the analogous notations for mixed discriminants of Hermitian matrices. 
\end{remark}

\section{Mixed discriminant and volume inequalities}\label{sec:mixed disc inequalities}

\subsection{From mixed discriminant to volume inequalities}

 We know from  Proposition \ref{hermitian vol} that any non-strict inequality for integrals of nef $(1,1)$-classes on a compact K\"ahler manifold implies the corresponding inequality for mixed discriminants of positive semidefinite Hermitian matrices. We show that the converse holds for certain types of inequalities.

\begin{proposition}\label{thm:disc=>Kahler}
Let $B$ be a multiset on $\llbracket n\rrbracket$ of cardinality $d$,
and let $I=(I_1,\dots,I_m)$ be an $r$-cover of $B$.
Suppose $\lambda$ is a positive constant such that, for any $d\times d$ positive semidefinite $\mathbb{C}$-Hermitian matrices $M_1,\dots,M_n$, and $A$, we have
\[
\lambda \D\Big[M^{ B}\Big]^r\D\Big[A^{\times d}\Big]^{m-r}
\leq
\D\Big[M^{ I_1} A^{\times(d-|I_1|)}\Big]\cdots \D\Big[M^{ I_m} A^{\times(d-|I_m|)}\Big].
\]
Then, for any nef $(1,1)$-class $\alpha$ and any collection of nef $(1,1)$-classes $\beta=(\beta_1,\dots,\beta_n)$ on a $d$-dimensional compact K\"ahler manifold, we have
\begin{equation*}\label{formula:Kaehler}
\lambda \left(\int \beta^B\right)^r\left(\int \alpha^d\right)^{m-r} \leq \left(\int  \beta^{I_1}\alpha^{d-|I_1|}\right) \cdots \left(\int  \beta^{I_m}\alpha^{d-|I_m|}\right).
\end{equation*}
\end{proposition}

\begin{proof}
We use the technique of Popovici \cite{popovici2016sufficient}, which relies on Yau's theorem on the existence of solutions to the Monge--Amp\`ere equation; see also \cite{lehmann2017correspondences}.

Since every nef class is the limit of K\"ahler classes, we may assume without loss of generality that $\alpha$ and each $\beta_i$ are K\"ahler forms on a compact K\"ahler manifold $X$. In particular, the product of any $d$ of them is a volume form on $X$.  
By Yau's theorem \cite{yau1977calabi}, there exists a K\"ahler form $\alpha_0$ representing the same cohomology class as $\alpha$ such that the following equality holds:
\[
\left(\int \beta^B \right)\alpha_0^d=\left(\int \alpha^d \right) \beta^B.
\]
We express each integral on the right-hand side of the desired
inequality as an integral with respect to the volume form
$\alpha_0^d$, using the fact that the ratio
$f_j := \beta^{I_j}\alpha_0^{d-|I_j|}/\alpha_0^d$
defines a smooth function on $X$:
\begin{equation*}\label{eq:start}
\int \beta^{I_j}\alpha^{d-|I_j|}
=\int f_j \alpha_0^d.
\end{equation*}
We recall that the mixed discriminant can be recovered by the wedge product of $(1,1)$-forms as follows: Fix a coordinate system $(z_1,\dots,z_d)$ centered at a point $x$ in $X$. A real $(1,1)$-form
\[
\omega(x)=\sqrt{-1}\sum_{1\leq i,j\leq d}M_{ij}(x)\,dz_i\wedge d\bar z_j
\]
is determined by its coefficient matrix $M\coloneq (M_{ij})$. Given $(1,1)$-forms $\omega_1,\dots,\omega_d$ with coefficient matrices $M_1,\dots,M_d$, respectively, we have
\[
\omega_1\cdots\omega_d=d!\D[M_1,\dots,M_d]\, \sqrt{-1}^d \,  dz_1\wedge d\bar z_1\wedge\cdots\wedge dz_d\wedge d\bar z_d.
\]
Applying the generalized H\"older inequality to the functions
$f_1,\ldots,f_m$ with respect to the volume form $\alpha_0^d$, we obtain
\[
\int \left(\prod_{j=1}^m f_j^{1/m}\right)\alpha_0^d 
 \leq \prod_{j=1}^m\left(\int f_j \alpha_0^d\right)^{1/m}.
\]
By the hypothesis applied to the integrand on the left-hand side, we get the pointwise inequality
\begin{equation*}
\lambda^{1/m}
\left(\int \beta^B\right)^{r/m}
\left(\int \alpha^d\right)^{-r/m}
=
\lambda^{1/m}\Big( \beta^B/\alpha_0^d\Big)^{r/m}
\leq \prod_{j=1}^m f_j^{1/m}.
\end{equation*}
Combining the two inequalities and raising to the power of $m$, we get the desired inequality.
\end{proof}

\begin{remark}
Under the same hypothesis on $\lambda$ as in Proposition \ref{thm:disc=>Kahler}, by using mass transport instead of complex Monge-Amp\`ere equations, we obtain a mixed volume inequality of the same form for convex bodies; see \cite[Theorem 5.9]{lehmann2017correspondences} for the case of the reverse Khovanskii-Teissier inequality.
Alternatively, for the mixed volume inequality, we may combine Proposition \ref{thm:disc=>Kahler} with  Proposition~\ref{prop:convex bodies are projective}, which states that convex body volume polynomials are projective volume polynomials.
\end{remark}

\subsection{A mixed discriminant inequality}

We prove an inequality that strengthens a  special case of Corollary~\ref{cor:discriminant main}. 
Its hypotheses are more restrictive than those of Corollary~\ref{cor:discriminant main}, requiring $B$ to be a set and $N$ to be of the form $(A,\ldots,A)$. Under these additional assumptions,  we get a stronger conclusion.
Let $d$ be a positive integer and $k$ be a nonnegative integer.

\begin{proposition}\label{prop:disc-exponents-1}
Let $I=(I_1,\dots,I_m)$ be an $r$-cover of a set $B$ of cardinality $d$. For any positive semidefinite Hermitian matrix $A$ and any collection of positive semidefinite Hermitian matrices $M=(M_i)_{i \in B}$ of size $\ell=d+k$, we have
\begin{equation*}
\D\left[M^{B},A^{[k]}\right]^r\D\left[A^{[\ell]}\right]^{m-r} \leq 
\D\left[M^{I_1},A^{[\ell-|I_1|]}\right] \cdots 
\D\left[M^{I_m},A^{[\ell-|I_m|]}\right].
\end{equation*}
\end{proposition}

Equality holds in Proposition~\ref{prop:disc-exponents-1} for every cover when $A$ is the identity matrix and each $M_i$ is the rank-one diagonal matrix whose $i$-th diagonal entry is $1$. In contrast, for the same matrices and $N=(A,\ldots,A)$, the inequality of Corollary~\ref{cor:discriminant main} is strict whenever $k$ is positive and the cover is nontrivial.

\begin{proof}
We proceed by induction on $d$, noting that the case $d=1$ is trivial. As in the introduction, we write $A_{\widehat{S}}$ for the principal submatrix obtained from $A$ by removing rows and columns indexed by $S$. Similarly, we write $M_{\widehat{S}}$ for the collection obtained from $M$ by replacing each $M_i$ in $M$ by $(M_i)_{\widehat{S}}$.

For the induction step, we use the fact that  positive definite matrices are dense in the cone of positive semidefinite matrices. Thus, we may assume without loss of generality that $A$ is nondegenerate. 
Note that simultaneous left and right multiplication by $A^{-1/2}$ allows us to further assume that $A$ is the identity matrix.
We identify $B$ with $\llbracket d\rrbracket$.
By the spectral theorem for Hermitian matrices, after conjugation by a suitable unitary matrix, we may further assume that a given matrix in $M$, say $M_1$, is a diagonal matrix. 

We begin with the case in which $M_1$ has at most rank $1$. In this case, without loss of generality, we may assume that $M_1$ is the diagonal matrix with diagonal entries $a,0,\ldots,0$ for some nonnegative number $a$. 

By definition of $\D$ and the assumption on $A$, for any positive semidefinite Hermitian matrices $C_1,\ldots,C_q$ of size $\ell$, we have
\begin{equation}\label{eq:expansionId}
\D\big[C_1,\dots,C_q,A^{[\ell-q]}\big]=\frac{q!}{\ell!}\sum_{|S|=\ell-q}\D\big[(C_1)_{\widehat{S}},\dots,(C_q)_{\widehat{S}}\big],
\end{equation}
where the sum is over all subsets $S \subseteq \llbracket \ell \rrbracket$ of cardinality $\ell-q$.
Similarly, by definition of $\D$ and the assumption on $M_1$, we have
\begin{equation}\label{eq:expansionRk1}
\D\big[M_1,C_2,\dots,C_\ell\big]=\frac{a}{\ell}\D\big[(C_2)_{\widehat{1}},\dots,(C_\ell)_{\widehat{1}}\big].
\end{equation}
We take the mixed discriminant in dimension zero to be $1$.

Let $J_1$ be the set of $j \in \llbracket m\rrbracket$ such that $1 \in I_j$, and let $J_2$ be the set of  $j \in \llbracket m\rrbracket$ such that $1 \notin I_j$.
Note that $J_1$ has cardinality $r$ and $J_2$ has cardinality $m-r$.
 Applying \eqref{eq:expansionId} and \eqref{eq:expansionRk1}, we see that the right-hand side of the desired inequality satisfies
\begin{equation}\label{eq:long}
\begin{aligned}
&\prod_{j=1}^m\D\left[M^{I_j},A^{[\ell-|I_j|]}\right]\\
&=\prod_{j\in J_1}\D\left[M^{I_j},A^{[\ell-|I_j|]}\right]\cdot\prod_{j\in J_2}\D\left[M^{I_j},A^{[\ell-|I_j|]}\right]\\
&=\left(\frac{a}{\ell}\right)^r\prod_{j\in J_1}\D\left[M_{\widehat{1}}^{I_j\setminus\{1\}},A_{\widehat{1}}^{[\ell-|I_j|]}\right]\cdot
\prod_{j\in J_2}\Bigg(\frac{|I_j|!}{\ell!}\sum_{|S|=\ell-|I_j|}\D\Big[M_{\widehat{S}}^{I_j}\Big]\Bigg)\\
&\geq\left(\frac{a}{\ell}\right)^r\prod_{j\in J_1}\D\left[M_{\widehat{1}}^{I_j\setminus\{1\}},A_{\widehat{1}}^{[\ell-|I_j|]}\right]\cdot
\prod_{j\in J_2}\Bigg(\frac{|I_j|!}{\ell!}\sum_{|S|=\ell-|I_j|, 1\in S}\D\Big[M_{\widehat{S}}^{I_j}\Big]\Bigg)\\
&=\left(\frac{a}{\ell}\right)^r\prod_{j\in J_1}\D\left[M_{\widehat{1}}^{I_j\setminus\{1\}},A_{\widehat{1}}^{[\ell-1-|I_j\setminus\{1\}|]}\right]\cdot
\prod_{j\in J_2}\Bigg(\frac{1}{\ell}\D\Big[M_{\widehat{1}}^{ I_j},A_{\widehat{1}}^{[\ell-1-|I_j|]}\Big]\Bigg)\\
&=\frac{a^r}{\ell^m}\prod_{j\in J_1}\D\left[M_{\widehat{1}}^{I_j\setminus\{1\}},A_{\widehat{1}}^{[\ell-1-|I_j\setminus\{1\}|]}\right]\cdot
\prod_{j\in J_2}\D\left[M_{\widehat{1}}^{ I_j},A_{\widehat{1}}^{[\ell-1-|I_j|]}\right].
\end{aligned}
\end{equation}
The inequality follows by retaining only the nonnegative summands
with $1\in S$. The next equality follows from \eqref{eq:expansionId} for matrices of size 
$\ell-1$.

Consider the $r$-cover of the set $B \setminus \{1\}$ given by the members $I_j\setminus\{1\}$ for $j\in J_1$ and $I_j$ for $j\in J_2$.
By the inductive hypothesis, we have
\begin{multline*}
\prod_{j\in J_1}\D\left[M_{\widehat{1}}^{I_j\setminus\{1\}},A_{\widehat{1}}^{[\ell-1-|I_j\setminus\{1\}|]}\right]\cdot
\prod_{j\in J_2}\D\left[M_{\widehat{1}}^{I_j},A_{\widehat{1}}^{[\ell-1-|I_j|]}\right]
\geq\D\left[M_{\widehat{1}}^{B \setminus \{1\}},A_{\widehat{1}}^{[k]}\right]^r\D\left[A_{\widehat{1}}^{[\ell-1]}\right]^{m-r}.
\end{multline*}
Therefore, the last line of \eqref{eq:long} is greater than or equal to
\begin{equation}\label{eq:lastchange}
\frac{a^r}{\ell^m}\D\left[M_{\widehat{1}}^{B \setminus \{1\}},A_{\widehat{1}}^{[k]}\right]^r\D\left[A_{\widehat{1}}^{[\ell-1]}\right]^{m-r}
=\D\left[M^B,A^{[k]}\right]^r\D\left[A^{[\ell]}\right]^{m-r},
\end{equation}
where the equality follows from \eqref{eq:expansionRk1} and the identities
\[
\D\left[A^{[\ell]}\right]=\frac1{\ell!}\quad\text{and}\quad\D\left[A_{\widehat{1}}^{[\ell-1]}\right]=\frac1{(\ell-1)!}.
\]
Combining \eqref{eq:long} and \eqref{eq:lastchange}, we obtain the desired inequality. This finishes the inductive step in case $\rk M_1 \le 1$.

Now we consider the general case, that is, when $M_1$ is of arbitrary rank. We may still assume that $A$ is the identity matrix  and $M_1$ is a diagonal matrix, say with nonnegative diagonal entries $a_1,\ldots, a_\ell$.  

Let $N_{i}$ be the diagonal matrix whose $i$-th diagonal entry is $a_i$ and whose other entries are zero. Then $M_1=\sum_{i=1}^{\ell} N_{i}$. Let $J_1$ and $J_2$ be as before, and for $j \in J_1$, define
\[
\D_{i,j}\coloneq\D\left[N_{i},M^{I_j\setminus\{1\}},A^{[\ell-|I_j|]}\right] \ \ \text{for $i=1,\ldots,\ell$.}
\]
We then have
\begin{equation}\label{eq:firstsplit}
\prod_{j=1}^m\D\left[M^{ I_j},A^{[\ell-|I_j|]}\right]^{1/r}=\prod_{j\in J_1}\left(\sum_{i=1}^\ell \D_{i,j}\right)^{1/r}\cdot\prod_{j\in J_2}\D\left[M^{I_j},A^{[\ell-|I_j|]}\right]^{1/r}.
\end{equation}
As $|J_1|=r$, by the H\"older inequality, we obtain
\[
\prod_{j\in J_1}\left(\sum_{i=1}^\ell\D_{i,j}\right)^{1/r}\geq\sum_{i=1}^\ell\left(\prod_{j\in J_1}\D_{i,j}\right)^{1/r}.
\]
Combining with \eqref{eq:firstsplit}, we have
\begin{equation}\label{eq:middle}
\begin{aligned}
\prod_{j=1}^m\D\left[M^{I_j},A^{[\ell-|I_j|]}\right]^{1/r}
&\geq\sum_{i=1}^\ell \left(\prod_{j\in J_1}\D_{i,j}\cdot\prod_{j\in J_2}\D\left[M^{I_j},A^{[\ell-|I_j|]}\right]\right)^{1/r}.
\end{aligned}
\end{equation}
By the rank-one case,  for each $1\leq i\leq \ell$, we know that
\[
\prod_{j\in J_1}\D_{i,j}\cdot\prod_{j\in J_2}\D\left[M^{ I_j},A^{[\ell-|I_j|]}\right]\geq\D\left[N_{i},M^{B \setminus \{1\}},A^{[k]}\right]^r\D\left[A^{[\ell]}\right]^{m-r}.
\]
Substituting this estimate into \eqref{eq:middle}, we obtain
\begin{align*}
\prod_{j=1}^m\D\left[M^{I_j},A^{[\ell-|I_j|]}\right]^{1/r}
&\geq\sum_{i=1}^\ell\D\left[N_{i},M^{B \setminus \{1\}},A^{[k]}\right]\D\left[A^{[\ell]}\right]^{(m-r)/r}\\
&=\D\left[M^{B},A^{[k]}\right]\D\left[A^{[\ell]}\right]^{(m-r)/r}.
\end{align*}
Raising to the power $r$ finishes the proof of the desired inequality.
\end{proof}

In contrast to Corollaries~\ref{cor:discriminant main}
 and~\ref{cor: mixed volumes}, the inequality corresponding to that in Proposition~\ref{prop:disc-exponents-1} does not hold for intersection products of nef divisors and mixed volumes of convex bodies.

\begin{example}\label{ex:octahedron}
Let $e_1,e_2,e_3$ be the standard basis vectors of $\RR^3$, and 
let $K\subseteq\RR^3$ be the octahedron with vertices $\pm(e_1+e_2)$, $\pm(e_1-e_2)$ and $\pm e_3$. Let $\mathscr{E}_i\subseteq\RR^3$ be the unit line segment in the direction $e_i$.  Then, we have
\[
\MV(\mathscr{E}_1,K,K)=\MV(\mathscr{E}_2,K,K)=\frac23,\quad\MV(\mathscr{E}_1,\mathscr{E}_2,K)=\frac13,\quad\MV(K,K,K)=\frac83.
\]
This violates the inequality
\[
\frac23\MV(\mathscr{E}_1,\mathscr{E}_2,K)\MV(K,K,K)
\leq \MV(\mathscr{E}_1,K,K)\MV(\mathscr{E}_2,K,K),
\]
which corresponds to the case $d=2$, $k=1$, and $r=1$ in Proposition \ref{prop:disc-exponents-1}.
\end{example}

\begin{remark}
When $d=2$ and  $r=1$,  Proposition \ref{prop:disc-exponents-1} coincides with \cite[Theorem~1.1]{artstein2014remarks}. Further examples demonstrating the failure of the corresponding inequalities for mixed volumes can be found in \cite[Subsection~4.1]{artstein2014remarks}.
\end{remark}

Since any convex body volume polynomial is a K\"ahler volume polynomial, 
Example~\ref{ex:octahedron} shows that the inequality analogous to that in Proposition \ref{prop:disc-exponents-1} need not hold for the coefficients of K\"ahler volume polynomials.
However, when $k=0$, the corresponding inequalities do hold for K\"ahler volume polynomials.
This is a special case of Theorem~\ref{thm introduction} for nef $(1,1)$-classes when $B$ is a set.

\begin{corollary}\label{cor:exponents1intersection}
Let $I=(I_1,\dots,I_m)$ be an $r$-cover of a set $B$ of cardinality $d$.
For any nef $(1,1)$-class $\alpha$ and a collection of nef $(1,1)$-classes $\beta=(\beta_i)_{i \in B}$ on a $d$-dimensional compact K\"ahler manifold,  we have
\[
\left(\int\beta^B \right)^r\left(\int \alpha^{[d]}\right)^{m-r} \leq  \left(\int \beta^{I_1}\alpha^{[d-|I_1|]} \right) \cdots \left(\int \beta^{I_m}\alpha^{[d-|I_m|]} \right).
\]
\end{corollary}
\begin{proof}
We apply Proposition~\ref{thm:disc=>Kahler} to the case $k=0$ of Proposition \ref{prop:disc-exponents-1}, absorbing the factorials into the constant $\lambda$.
\end{proof}

\begin{remark}
One can show that, in fact, the inequality of Corollary \ref{thm:intro-volume-polynomials} holds for all homogeneous stable polynomials with nonnegative coefficients \cite{HMWX-stable}. 
Since every nonzero Hermitian volume polynomial is stable,
Proposition~\ref{thm:disc=>Kahler} transfers this result to compact K\"ahler
manifolds, yielding another proof of Theorem~\ref{thm introduction} for projective varieties in characteristic zero.
\end{remark}

\section{Volume inequalities for projective varieties}\label{sec:volproj}

Let $X$ be a $d$-dimensional projective variety over an algebraically closed field $\K$.
The goal of this section is to prove the inequality of Corollary \ref{cor:exponents1intersection} for nef divisors on $X$ (Proposition~\ref{thrm cover of set}). 

\subsection{Volume defined by global sections}
We recall some definitions and properties concerning the volume of divisors defined by global sections. Unless otherwise specified, a divisor will mean an integral Cartier divisor. Divisors with rational or real coefficients will be referred to as $\QQ$-divisors or $\RR$-divisors, respectively.

\begin{definition}
Given a Cartier divisor $\delta$ on $X$, its \emph{volume} is defined by
\[
\vol(\delta)\coloneq\limsup_{q\rightarrow\infty}\frac{h^0(X,q\delta)}{q^d/d!}.
\]
\end{definition}

A Cartier divisor $\delta$ is called \emph{big} if $\vol(\delta)>0$. Let $N^1(X)$ be the group of numerical equivalence classes of divisors. Denote by $N^1(X)_\RR=N^1(X)\otimes_\ZZ\RR$ the vector space of divisor classes with real coefficients.
It is well known that the volume $\vol(\delta)$ depends only on the numerical class of $\delta$, and extends to a continuous function on $N^1(X)_\RR$; see \cite[Section 2.2]{lazarsfeldPosI}.

Inside the vector space $N^1(X)_\RR$, we have the nef cone and the pseudoeffective cone of $\RR$-divisor classes, which satisfy
\[
\Nef^1(X)\subseteq\Psef^1(X)\subseteq N^1(X)_\RR.
\]
Recall that an $\RR$-divisor class $\delta$ is called \emph{nef} if $\delta\cdot C\geq 0$ for any irreducible curve $C\subseteq X$, and $\delta$ is called \emph{psef} if $\delta$ is a limit of effective $\RR$-divisor classes. The interior of $\Nef^1(X)$, denoted by $\Amp^1(X)$, is called \emph{the ample cone}, which exactly consists of ample $\RR$-divisor classes. The interior of $\Psef^1(X)$, denoted by $\bg^1(X)$, is called the \emph{big cone}. The big cone can be characterized by the volume function on $N^1(X)_\RR$ by the condition
\[
\delta\in\bg^1(X)\ \ \text{if and only if} \ \ \vol(\delta)>0.
\]
In particular, $\vol(\delta)=0$ when $\delta$ lies on the boundary of $\Psef^1(X)$.

\begin{definition}
For $\delta,\beta\in\Psef^1(X)$, the \emph{psef threshold} of $\delta$ with respect to $\beta$ is defined by
\[
r_{\psef}(\delta,\beta)=\sup\{t\geq 0\mid\delta-t\beta\in\Psef^1(X)\}.
\]
\end{definition}

Let $\delta$ be a big divisor on $X$. The \emph{stable base locus} $\mathbf B(\delta)$ of $\delta$ is defined as the intersection of the base loci $\bs(|q\delta|)$ for all $q\in\ZZ_{>0}$, which is equivalent to working with all sufficiently divisible $q$. Hence the stable base locus extends to any $\QQ$-divisor. Let $\eta$ be a sufficiently small ample $\QQ$-divisor. The augmented base locus $\mathbf B_+(\delta)$ of $\delta$ is given by $\mathbf B(\delta-\eta)$. It is known that $\mathbf B_+(\delta)$ is independent of $\eta$ provided that it is sufficiently small, depends only on the numerical class of $\delta$, and makes sense for divisor classes with real coefficients.

\begin{definition}
Let $V\subseteq X$ be a subvariety of dimension $p$ and let $\delta$ be a divisor on $X$. The \emph{restricted volume} $\vol_{X|V}(\delta)$ of $\delta$ is defined by
\[
\vol_{X|V}(\delta)\coloneq\limsup_{q\rightarrow\infty}\frac{\dim\im\bigl(H^0(X,q\delta)\rightarrow H^0(V,q\delta|_V)\bigr)}{q^p/p!}.
\]
\end{definition}

On the cone of divisor classes with $V\nsubseteq\mathbf B_+(\delta)$, the restricted volume also extends continuously to real divisor classes; see \cite[Theorem~A]{elmnp09restrictedvolume}, and in arbitrary characteristic, \cite[Section 4]{lazamustNewtonOkou}. Denote by $\vol_V(-)$ the volume function on $V$. By definition, $\vol_{X|V}(\delta)\leq\vol_V(\delta)$, where $\vol_V(\delta)$ is the volume of the pullback of $\delta$ to $V$.
We have $\vol_{X|V}=\vol_X$ when $V=X$, and if $\delta$ is ample, we have
\[
\vol_{X|V}(\delta)=\vol_V(\delta)=(\delta^p)_V.
\]

\begin{definition}
Let $V\subseteq X$ be a subvariety of dimension $p$ and let $\delta$ be a big divisor on $X$ with $V\nsubseteq\mathbf B(\delta)$. For sufficiently divisible positive integers $q$, let $\pi_q:X_q\rightarrow X$ be the blow-up along the base ideal of $|q\delta|$ with $\pi_q^*|q\delta|=|\mu_q|+\varepsilon_q$, where $\mu_q$ is free and $\varepsilon_q$ is the fixed part. The asymptotic intersection number $\|\delta^p\cdot V\|$ is defined by
\[
\|\delta^p\cdot V\|=\limsup_{q\rightarrow\infty}\frac{(\mu_q^p)_{V_q}}{q^p},
\]
where $V_q$ is the strict transform of $V$ under $\pi_q$.
\end{definition}

By \cite[Remark~2.9]{elmnp09restrictedvolume}, the limsup taken along sufficiently divisible multiples is a limit. The definition of the asymptotic intersection number extends naturally to $\QQ$-divisors. We will use the following two lemmas. The first follows from \cite[Theorem~2.13]{elmnp09restrictedvolume} over $\mathbb{C}$ and from \cite[Remark~3.6]{lazamustNewtonOkou} in arbitrary characteristic. The second follows from \cite[Theorem A and Theorem B]{BFJ09}; see also \cite[Corollary~C and Corollary~4.27]{lazamustNewtonOkou}.
 
\begin{lemma}\label{restrvol1}
For a big $\QQ$-divisor class $\delta$ on $X$ and a $p$-dimensional subvariety $V\subseteq X$ with $V\nsubseteq\mathbf B_+(\delta)$, we have 
\[
\vol_{X|V}(\delta)=\|\delta^p\cdot V\|.
\]
\end{lemma}

\begin{lemma}\label{restrvol2}
The volume function $\vol$ is continuously differentiable in the big cone. More precisely, given a big $\RR$-divisor class $\delta$ and an irreducible Cartier divisor $H$ on $X$, for any $t_0$ such that $\delta+t_0[H]$ is big, we have
\[
\left.\frac{d}{dt}\right|_{t=t_0}\vol(\delta+t[H])=d\vol_{X|H}(\delta+t_0[H]).
\]
\end{lemma}

Under these hypotheses, the derivative formula shows that $\vol_{X|H}(\delta+t_0[H])$ depends only on the numerical classes of $\delta$ and $H$.

\subsection{Volume inequalities for the cover of a set}

We prove Theorem~\ref{thm introduction} for nef divisors on a projective variety when $B$ is a set. 

\begin{proposition}\label{thrm cover of set}
Let $I=(I_1,\ldots,I_m)$ be an $r$-cover of the $d$-element set $B=\llbracket d \rrbracket$. 
For any nef divisors $\alpha,\beta_1,\ldots,\beta_d$  on a $d$-dimensional projective variety, we have
\[
\Big(\beta_1\beta_2 \cdots \beta_d\Big)^r
\Big(\alpha^{[d]}\Big)^{m-r}
\leq
\Big(\beta^{I_1}\alpha^{[d-|I_1|]}\Big)
\cdots
\Big(\beta^{I_m}\alpha^{[d-|I_m|]}\Big),
\]
where $\beta$ is the collection  $(\beta_1,\ldots,\beta_d)$.
\end{proposition}
\begin{proof}
We may suppose that $m>r$. We prove the result by induction on $d$. We assume that the desired inequality holds in dimension $d-1$. Note that when $X$ is a curve, the inequality is obvious.

 We note that it is enough to prove the statement when $X$ is smooth and $\alpha,\beta_1,\dots,\beta_d$ are very ample divisors. Indeed, consider de Jong's alteration $f:Y\rightarrow X$, with $Y$ smooth and projective \cite{de1996smoothness}, and let $a=[\K(Y):\K(X)]$.
By the projection formula, each intersection of the pulled-back classes equals $a$ times the original intersection number. As the pullback of a nef divisor is nef, the inequality for $Y$ and the pulled-back divisors gives a rescaling by $a^m$ of the inequality for $X$. 
Furthermore, since base change does not change intersection numbers, replacing $\K$ by an uncountable algebraically closed extension, we may assume that the defining field $\K$ is uncountable.
For a subvariety $V\subseteq X$ of dimension $p$, we set
\[
\vol_{X|V}\big(\delta^{[p]}\big)\coloneq\frac1{p!}\vol_{X|V}\big(\delta\big) \ \ \text{and} \ \  \big\|\delta^{[p]}\cdot V\big\| \, \coloneq\frac1{p!} \big\|\delta^p\cdot V\big\|.
\]

Our basic observation towards Proposition~\ref{thrm cover of set} is as follows: if $\alpha$ is a nef $\RR$-divisor class lying on the boundary of $\Psef^1(X)$, then the inequality obviously holds. In the general case when $\alpha$ is big and nef, we shall deform $\alpha$ to the boundary of $\Psef^1(X)$ and apply the induction in the deformation process.

Let $\alpha_t=\alpha+t\beta_1$, and let $t_1=r_{\psef,X}(\alpha,\beta_1)$ be the psef threshold of $\alpha$ with respect to $\beta_1$. We have $\alpha_t\in\bg^1(X)$ for any $t>-t_1$ and $\vol(\alpha_{-t_1})=0$.
Take a smooth prime divisor $H\in|\beta_1|$. 
By Lemma \ref{restrvol2}, we have 
\[
\Big(\alpha^{[d]}\Big)=\vol\big(\alpha^{[d]}\big)-\vol\big(\alpha_{-t_1}^{[d]}\big)=\int_{-t_1}^{0}\vol_{X|H}\big(\alpha_t^{[d-1]}\big)\,dt=\lim_{s\searrow-t_1}\int_s^0\vol_{X|H}\big(\alpha_t^{[d-1]}\big)\,dt.
\]
Then the left-hand side of the inequality in Proposition~\ref{thrm cover of set} becomes
\begin{equation}\label{lhs}
\Big(\beta_1\cdots\beta_d\Big)^r\Big(\alpha^{[d]}\Big)^{m-r}=\lim_{s\searrow-t_1}\left(\int_s^0 \Big(\beta_2\cdots\beta_d\Big)_H^{r/(m-r)} \cdot \vol_{X|H}\big(\alpha_t^{[d-1]}\big)\,dt\right)^{m-r},
\end{equation}
where $(\beta_2\cdots\beta_d)_H$ denotes the intersection number of the restrictions on $H$.

Fix $s\in(-t_1,0)$.
Let $H\in|\beta_1|$ be chosen very generally, so that the restriction of $\alpha_{-t_1}$ to $H$ is pseudoeffective. It follows that $\alpha_t|_H=(\alpha_{-t_1}+(t+t_1)\beta_1)|_H$ is big for any $t>-t_1$. We can also require that $H\nsubseteq\mathbf B_+(\alpha_s)$. Then $H\nsubseteq\mathbf B_+(\alpha_t)$ for any $t>s$, as $\alpha_t=\alpha_s+(t-s)\beta_1$ is more positive than $\alpha_s$.

By continuity of $\vol_{X|H}(-)$, to estimate the integrand in \eqref{lhs}, it is enough to estimate
\[
\Big(\beta_2\cdots\beta_d\Big)_H^{r/(m-r)}\cdot \vol_{X|H}\big(\alpha_t^{[d-1]}\big) \ \ \text{for $t\in(s,0)\cap\QQ$}.
\]
By Lemma~\ref{restrvol1}, we have $\vol_{X|H}\big(\alpha_t^{[d-1]}\big)=\big\|\alpha_t^{[d-1]}\cdot H\big\|$.
By definition of asymptotic intersection numbers, we may take a sequence of blow-ups $\pi_q:X_q\rightarrow X$ realizing $\big\|\alpha_t^{[d-1]}\cdot H\big\|$.
For $t\in(s,0)\cap\QQ$, we have
\[
\Big(\beta_2\cdots\beta_d\Big)_H^{r/(m-r)} \cdot \Big\|\alpha_t^{[d-1]}\cdot H\Big\|=\limsup_{q\rightarrow\infty}\Big(\pi_q^*\beta_2\cdots\pi_q^*\beta_d\Big)_{H_q}^{r/(m-r)} \cdot \Big(\alpha_{t,q}^{[d-1]}\Big)_{H_q},
\]
where $H_q$ is the strict transform of $H$ by $\pi_q$ and $\alpha_{t,q}$ is $1/q$ times the free part of the pullback of $q\alpha_t$ with $q$ sufficiently divisible.

Note that the collection given by $I_j\setminus\{1\}$ with $1\in I_j$ and $I_j$ with $1\notin I_j$ constitutes an $r$-cover of the set $B \setminus \{1\}$. 
By applying our induction hypothesis to the nef classes $\alpha_{t,q}$ and $\pi_q^*\beta_2,\dots,\pi_q^*\beta_d$ on the hypersurface $H_q$, we obtain the upper bound
\begin{multline}\label{induc1}
\Big(\beta_2\cdots\beta_d\Big)_H^{r/(m-r)} \cdot \vol_{X|H}\big(\alpha_t^{[d-1]}\big)\\
\leq\limsup_{q\rightarrow\infty}\prod_{j \in J_1}\Big(\pi_q^*\beta^{I_j\setminus\{1\}}\alpha_{t,q}^{[d-|I_j|]}\Big)_{H_q}^{1/(m-r)} \cdot \prod_{j \in J_2}\Big(\pi_q^*\beta^{I_j}\alpha_{t,q}^{[d-|I_j|-1]}\Big)_{H_q}^{1/(m-r)},
\end{multline}
where $J_1$ is the set of $j \in \llbracket m\rrbracket$ such that $1 \in I_j$, and $J_2$ is the set of  $j \in \llbracket m\rrbracket$ such that $1 \notin I_j$.

We estimate the factors in the two products for $J_1$ and $J_2$ in \eqref{induc1} as follows. For each $\beta_i$ with $i\neq1$, take a very general smooth hypersurface $H_i\in|\beta_i|$.  Set $Y_j \coloneq \bigcap_{i\in I_j\setminus\{1\}}H_i$. As the $H_i$ are taken very generally, $\alpha_t|_{Y_j\cap H}$ is big, and $Y_j\cap H$ is not contained in $\mathbf B_+(\alpha_t)$ for any $t>s$. 
Then, for $t\in(s,0)\cap\QQ$, a factor in the product for $J_1$ satisfies
\[
\limsup_{q\rightarrow\infty}\Big(\pi_q^*\beta^{I_j\setminus\{1\}}\alpha_{t,q}^{[d-|I_j|]}\Big)_{H_q}
=\vol_{X|Y_j\cap H}\big(\alpha_t^{[d-|I_j|]}\big)
\leq\vol_{X|Y_j\cap H}\big(\alpha^{[d-|I_j|]}\big)=\Big(\beta^{I_j}\alpha^{[d-|I_j|]}\Big),
\]
where in the first equality we apply Lemma \ref{restrvol1} and use the fact that the intersection of $\pi_q^*\beta^{I_j\setminus\{1\}}$ with $H_q$ is represented by the strict transform of $Y_j\cap H$. In the middle inequality, we apply the monotonicity of restricted volumes, and in the last equality, we use the fact that $\alpha$ is ample.
For the factors in the product for $J_2$, set $Z_j\coloneq \bigcap_{i\in I_j}H_i$. Then
\begin{equation*}
\limsup_{q\rightarrow\infty}\Big(\pi_q^*\beta^{I_j}\alpha_{t,q}^{[d-|I_j|-1]}\Big)_{H_q}=\vol_{X|Z_j\cap H}\big(\alpha_t^{[d-|I_j|-1]}\big)\leq\vol_{Z_j|Z_j\cap H}\big(\alpha_t^{[d-|I_j|-1]}\big),
\end{equation*}
where the inequality follows from the definition of restricted volumes. Thus, the left-hand side of \eqref{induc1} is bounded above by the product 
\begin{equation*}
\prod_{j \in J_1}\Big(\beta^{I_j}\alpha^{[d-|I_j|]}\Big)^{1/(m-r)}\prod_{j \in J_2}\vol_{Z_j|Z_j\cap H}\big(\alpha_t^{[d-|I_j|-1]}\big)^{1/(m-r)}.
\end{equation*}
By continuity, the same estimate holds for any $t\in(s,0)$. 
Since   $J_2$ has cardinality $m-r$, H\"older's inequality gives
\begin{equation*}
\int_s^0\prod_{j \in J_2}\vol_{Z_j|Z_j\cap H}\big(\alpha_t^{[d-|I_j|-1]}\big)^{1/(m-r)}\,dt\leq\prod_{j \in J_2}\left(\int_s^0\vol_{Z_j|Z_j\cap H}\big(\alpha_t^{[d-|I_j|-1]}\big)\,dt\right)^{1/(m-r)}.
\end{equation*}
By combining \eqref{lhs} and the above two displayed formulas, 
we obtain
\[
\Big(\beta_1\cdots\beta_d\Big)^r\Big(\alpha^{[d]}\Big)^{m-r}\leq\prod_{j \in J_1}\Big(\beta^{I_j}\alpha^{[d-|I_j|]}\Big) \cdot \prod_{j\in J_2}\int_{-t_1}^0\vol_{Z_j|Z_j\cap H}\big(\alpha_t^{[d-|I_j|-1]}\big)\,dt.
\]
To finish the proof, it is enough to verify that, for any $j \in J_2$, we have
\begin{equation*}\label{volrest}
\int_{-t_1}^0\vol_{Z_j|Z_j\cap H}\big(\alpha_t^{[d-|I_j|-1]}\big)\,dt\leq\big(\beta^{I_j}\alpha^{[d-|I_j|]}\big).
\end{equation*}
To this end, note that the restriction of $\alpha_{-t_1}$ to $Z_j$ is pseudoeffective, as $Z_j$ is very general. Denote by  $t_2=r_{\psef,Z_j}(\alpha|_{Z_j},\beta_1|_{Z_j})$ the psef threshold of $\alpha|_{Z_j}$ with respect to $\beta_1|_{Z_j}$ on $Z_j$. Then $t_2\geq t_1$. Using Lemma \ref{restrvol2} on $Z_j$, we get the desired conclusion
\begin{align*}
\Big(\beta^{I_j}\alpha^{[d-|I_j|]}\Big)=\vol_{Z_j}\big(\alpha^{[d-|I_j|]}\big)
&=\int_{-t_2}^0\vol_{Z_j|Z_j\cap H}\big(\alpha_t^{[d-|I_j|-1]}\big)\,dt\\
&\geq\int_{-t_1}^0\vol_{Z_j|Z_j\cap H}\big(\alpha_t^{[d-|I_j|-1]}\big)\,dt. \qedhere
\end{align*}
\end{proof}

This proves Theorem~\ref{thm introduction} for nef divisors on a projective variety when $B$ is a subset of $E$.

\section{Proof of Theorem~\ref{thm introduction}}\label{sec:setmultiset}

In this section, we prove Theorem~\ref{thm introduction}  when $B$ is any multiset on $E$. As before, we work over a fixed algebraically closed field $\K$. 
We recall two operations on polynomials. 
\begin{enumerate}[(1)]\itemsep 2 pt
\item Let $f$ be a polynomial in $n$ variables $x=(x_1,\ldots,x_n)$. We write $f=\sum_A a_A x^A$, where the sum is over the multisets $A$ on $\llbracket n\rrbracket$.
The \emph{upper truncation} of $f$ with respect to a multiset $B$ on $\llbracket n \rrbracket$ is the partial sum 
$f_{\leq B}\coloneq \sum_{A\subseteq B} a_A x^A$.
\item Let $g$ be a polynomial in $n+1$ variables $(y_0,x)=(y_0,x_1,\ldots,x_n)$, and suppose that the degree of $g$ in $y_0$ is at most $q$. The \emph{polarization operator} $\Pi^\uparrow$ with respect to $y_0^q$ is the linear operator defined by its values on monomials 
\[
\Pi^\uparrow\!\left(y_0^a \, x^A\right)
\coloneq
\binom{q}{a}^{-1} e_a(y_1,\dots,y_q) \ x^A,
\]
for all nonnegative $a \le q$ and multisets $A$ on $\llbracket n \rrbracket$, 
where $e_a$ denotes the $a$-th elementary symmetric polynomial in the new variables $(y_1,\ldots,y_q)$. The \emph{polarization} of $g$ with respect to $y_0^q$ is the polynomial $\Pi^\uparrow g$ in variables $(y_1,\ldots,y_q,x_1,\ldots,x_n)$. 
\end{enumerate}
These operations preserve the class of realizable projective volume polynomials over $\K$; see \cite[Corollary~3.3]{LinearOperator} and \cite[Proposition~4.1]{LinearOperator}.

\begin{theorem}\label{thm:truncation}
Suppose that  $f$ is a realizable projective volume polynomial over $\K$.
\begin{enumerate}[(1)]\itemsep 2 pt
\item The upper truncation $f_{\leq B}$ with respect to any multiset $B$ is a realizable projective volume polynomial over $\K$.
\item The polarization $\Pi^\uparrow f$ with respect to any valid power of a variable is a realizable projective volume polynomial over $\K$.
\end{enumerate}
\end{theorem}

It follows from Theorem~\ref{thm:truncation} and continuity that both upper truncation and polarization preserve the class of projective volume polynomials over $\K$. 
We record the following corollary for intersection numbers of semiample $\RR$-divisors on projective varieties.\footnote{By a semiample $\RR$-divisor, we mean an $\RR$-divisor that becomes a semiample $\QQ$-divisor after multiplication by some positive real number. This convention allows a common positive scalar to be absorbed into the representing classes in Corollary~\ref{cor:truncation} and includes the zero class.}

\begin{corollary}\label{cor:truncation}
Let $\delta_0$ be a semiample $\RR$-divisor and let $\gamma=(\gamma_1,\dots,\gamma_n)$ be a collection of semiample $\RR$-divisors on a $d$-dimensional projective variety $X$. For any positive integer $q \le \dim X$, there are  collections of semiample $\RR$-divisors $\delta=(\delta_1,\dots,\delta_q)$ and $\eta=(\eta_1,\ldots,\eta_n)$ on a $d$-dimensional projective variety $Y$ with the following properties:
\begin{enumerate}[(1)]\itemsep 2pt
\item For any multiset $A$ on $\llbracket q \rrbracket$ and multiset $B$ on $\llbracket n \rrbracket$ with $|A|+|B|=d$, we have
\[
\Big(\delta^{A} \, \eta^{B} \Big)=0 \ \ \text{if $A$ is not a subset of $\llbracket q \rrbracket$}.
\]
\item For any $0\leq a\leq q$ and any multiset $B$ on $\llbracket n \rrbracket$ with $a+|B|=d$, we have
\[
\Big(\delta_0^{a} \gamma^{B} \Big)=\Big((\delta_1+\dots+\delta_q)^{a} \, \eta^{B}  \Big),
\]
where the left-hand side is computed in $X$ and the right-hand side is computed in $Y$.
\end{enumerate}
It follows that, for any $0\leq a\leq q$ and any multiset $B$ on $\llbracket n \rrbracket$ with $a+|B|=d$, we have
\[
\Big(\delta_0^{[a]} \gamma^{B} \Big)=\Big((\delta_1+\dots+\delta_q)^{[a]} \, \eta^{B}  \Big)=\Big(e_a(\delta_1,\ldots,\delta_q) \, \eta^{B}  \Big),
\]
where $e_a$ is the $a$-th elementary symmetric polynomial in $q$ variables.
\end{corollary}

\begin{proof}
By rescaling, we may suppose that $\delta_0$ is a semiample $\mathbb{Q}$-divisor and $\gamma$ is a collection of semiample $\mathbb{Q}$-divisors on $X$. Consider the realizable projective volume polynomial
\[
f(y_0,x_1,\dots,x_n)=\frac{1}{d!}\deg_X \left((y_0\delta_0+x_1\gamma_1+\cdots+x_n\gamma_n)^d\right).
\]
 By Theorem \ref{thm:truncation}, truncating $f$ by removing all terms involving $y_0^{q+1}$ or higher exponents, we obtain a realizable projective volume polynomial $g$. Next, we consider the  polarization $\Pi^\uparrow g$ with respect to $y_0^q$, which has $q$ new variables $y_1,\ldots,y_q$. Up to rescaling, the semiample divisors witnessing the fact that $\Pi^\uparrow g$ is a realizable projective volume polynomial satisfy the desired properties.
 \end{proof}

We now prove Theorem~\ref{thm introduction} for nef divisors on projective varieties. Later in this section, we use it to prove Theorem~\ref{thm introduction} for nef $(1,1)$-classes on compact K\"ahler manifolds, extending Corollary~\ref{cor:exponents1intersection}.

\begin{proof}[Proof of Theorem~\ref{thm introduction} for nef divisors on projective varieties]
Since every nef divisor class is the limit of semiample $\QQ$-divisor classes, we may assume that all the divisors $\alpha$ and $\beta_i$ are semiample $\QQ$-divisors. Let $n$ be the cardinality of $E$, and identify $E$ with the set $\llbracket n \rrbracket$. 

We may discard the indices $i$ with $B(i)=0$. We proceed by induction on the cardinality of the set $\mathcal S(B)\coloneq\{i\in \llbracket n \rrbracket:B(i)>1\}$. If $\mathcal S(B)$ is empty, then $B$ is a subset of $\llbracket n \rrbracket$ and the statement to be proved is precisely Proposition~\ref{thrm cover of set}.

Suppose $\mathcal S(B)$ is nonempty. Without loss of generality, we may assume that $n\in\mathcal S(B)$. 
Applying Corollary \ref{cor:truncation} with $q=B(n)$ to the nef divisors
\[
\delta_0\coloneq \beta_n \ \ \text{and} \ \ \gamma=(\gamma_1,\ldots,\gamma_{n-1},\gamma_n)=(\beta_1,\ldots,\beta_{n-1},\alpha),
\]
 we know that there is a projective variety $Y$ and collections of semiample $\RR$-divisors $\delta=(\delta_1,\dots,\delta_q)$ and $\eta=(\eta_1,\dots,\eta_{n-1},\eta_n)$ on $Y$ such that for any $j \in \llbracket m \rrbracket$, we have 
\begin{align}
\Big(\beta^{[B]}\Big)\Big(\alpha^{[d]}\Big)
&=\Big(\Big(\prod_{i=1}^{n-1}\eta_i^{[B(i)]}\Big)
\delta_1\cdots\delta_q\Big)\Big(\eta_n^{[d]}\Big),
\label{eq:polarization-B}\\
\Big(\beta^{[I_j]}\alpha^{[d-|I_j|]}\Big)
&=\Big(\Big(\prod_{i=1}^{n-1}\eta_i^{[I_j(i)]}\Big)
e_{a_j}(\delta_1,\ldots,\delta_q)
\eta_n^{[d-|I_j|]}\Big),
\label{eq:polarization-I}
\end{align}
 where $a_j \coloneq I_j(n) \le q = B(n)$. 
Applying the Arithmetic Mean--Geometric Mean inequality to the right-hand side of \eqref{eq:polarization-I}, we have
\begin{equation}\label{eq:AM-GM}
\frac{1}{\binom{q}{a_j}}\Big(\beta^{[I_j]}\alpha^{[d-|I_j|]}\Big)
\geq \prod_{S}\Big(\Big(\prod_{i=1}^{n-1}\eta_i^{[I_j(i)]}\Big) \, \delta^S \, \eta_n^{[d-|I_j|]}\Big)^{\frac{1}{\binom{q}{a_j}}},
\end{equation}
where the outer product is over all subsets $S$ of $\llbracket q \rrbracket$ of cardinality $a_j$.

The multiset $B$ lifts to a multiset $\underline{B}$ on $\underline{E} \coloneq\{1,\dots,n-1,1^*,\dots, q^*\}$ defined by
\[
\underline{B}(i)=B(i) \ \ \text{for $1\leq i<n$} \ \ \text{and} \ \ \underline{B}(a^*)=1 \ \ \text{for $1\leq a\leq q$}.
\]
Similarly, for any $j \in \llbracket m \rrbracket$ and subset $S$ of $\llbracket q \rrbracket$ of cardinality $a_j$, we define a multiset $\underline{I}_{j,S}$ on $\underline{E}$ by
\[
\underline{I}_{j,S}(i)=I_j(i) \ \ \text{for $1\leq i< n$} \ \ \text{and} \ \ 
\underline{I}_{j,S}(a^*)=\begin{cases}1&\text{if $a\in S$},\\0&\text{if $a\notin S$}.\end{cases}
\]
Let $N$ be the least common multiple of $\binom{q}{a_j}$ for $j \in \llbracket m \rrbracket$, and denote $N/\binom{q}{a_j}$ by $q_j$. We consider the collection of multisets
\[
\underline{I}=\Big(\underline{I}_{j,S}^{\times q_j} \Bigm| \text{$j \in \llbracket m\rrbracket$ and $S$ is a subset of $\llbracket q \rrbracket$ of cardinality $a_j$}\Big),
\]
where each multiset $\underline{I}_{j,S}$ appears $q_j$ times. It is straightforward to check that $\underline{I}$ is an $rN$-cover of $\underline{B}$ by $mN$ multisets. 

Since $\mathcal{S}(\underline{B})=\mathcal{S}(B) \setminus \{n\}$, the induction hypothesis applies to $\underline{I}$ and nef divisors on $Y$:
\[
\binom{\underline{B}}{\underline{I}}\Big(\Big(\prod_{i=1}^{n-1}\eta_i^{[B(i)]}\Big)\delta_1\cdots\delta_q\Big)^{rN}\Big(\eta_n^{[d]}\Big)^{(m-r)N}\\
\leq \prod_{j, S}\Big(\Big(\prod_{i=1}^{n-1}\eta_i^{[I_j(i)]}\Big) \, \delta^S \, \eta_n^{[d-|I_j|]}\Big)^{q_j},
\]
where the outer product on the right-hand side is over all $j \in \llbracket m \rrbracket$ and subsets $S$ of $\llbracket q \rrbracket$ of cardinality $a_j$.
Taking the $N$-th root and combining with the inequality \eqref{eq:AM-GM}, we have
\begin{equation}\label{eq:Nthroot}
\binom{\underline{B}}{\underline{I}}^{\frac{1}{N}}\Big(\Big(\prod_{i=1}^{n-1}\eta_i^{[B(i)]}\Big)\delta_1\cdots\delta_q\Big)^{r}\hspace{-1mm}\Big(\eta_n^{[d]}\Big)^{m-r}
\leq \prod_{j=1}^m \frac{1}{\binom{q}{a_j}} \Big(\beta^{[I_j]}\alpha^{[d-|I_j|]}\Big).
\end{equation}
Notice that the constant on the left-hand side satisfies
\begin{equation*}
\binom{\underline{B}}{\underline{I}}^{\frac{1}{N}}
=\prod_{j,S}\binom{\underline{B}}{\underline{I}_{j, S}}^{\frac{1}{\binom{q}{a_j}}}
=\prod_{j=1}^m\frac{1}{\binom{q}{a_j}}\binom{B}{I_{j}}
=\binom{B}{I} \prod_{j=1}^m\frac{1}{\binom{q}{a_j}}.
\end{equation*}
Thus, using \eqref{eq:polarization-B}, we can simplify \eqref{eq:Nthroot} to the desired inequality
\[
\binom{B}{I}\left(\beta^{[B]}\right)^{r}\left(\alpha^{[d]}\right)^{(m-r)}
\leq \prod_{j=1}^m\left(\beta^{[I_j]}\alpha^{[d-|I_j|]}\right). \qedhere
\]
\end{proof}

\begin{example}\label{ex:rKT violation}
Consider the cubic polynomial in three variables
\[
f=14x_1^3+6x_1^2x_2+24x_1^2x_3+12x_1x_2x_3+6x_1x_3^2+3x_2x_3^2.
\]
Using \cite[Theorem 2.25]{branden2020lorentzian}, one can check that $f$ is a Lorentzian polynomial.
The fact that $f$ is not a projective volume polynomial follows from the reverse Khovanskii--Teissier inequality, a special case of  Theorem~\ref{thm introduction}; see \cite[Section 3.2]{HuhVolume} for details.
\end{example}

\begin{example}\label{ex:new inequality}
Consider the quartic polynomial in five variables
\begin{multline*}
f=
x_0^4+x_0^3(x_1+x_2+x_3+x_4)
+x_0^2\left(\frac{111}{100}x_1x_2+\frac{113}{100}(x_1x_3+x_1x_4+x_2x_3+x_2x_4)+\frac1{100}x_3x_4\right)\\
+x_0\left(\frac{117}{100}(x_1x_2x_3+x_1x_2x_4)+\frac{53}{5000}(x_1x_3x_4+x_2x_3x_4)\right)+\frac{21}{2000}x_1x_2x_3x_4.
\end{multline*}
Using \cite[Theorem 2.25]{branden2020lorentzian}, one can check that $f$ is a Lorentzian polynomial. 
Furthermore, $f$ satisfies the reverse Khovanskii--Teissier inequalities, including those obtained after applying any nonnegative linear change of coordinates; there are $295$ distinct cases to check. 
However, 
\[
(f^2\cdot\partial^{\{1,2,3,4\}}f)(1,0,0,0,0)=\frac{21}{2000}
\ \ \text{and} \ \ 
(\partial^{\{1\}}f\cdot\partial^{\{2\}}f\cdot\partial^{\{3,4\}}f)(1,0,0,0,0)=\frac1{100}.
\]
This contradicts the inequality in Corollary~\ref{thm:intro-volume-polynomials} for the partition of  $\{1,2,3,4\}$ into $\{1\},\{2\},\{3,4\}$, so $f$ cannot be a projective volume polynomial. 
Thus, Theorem~\ref{thm introduction} imposes new constraints on volume polynomials that are not implied by the combination of Lorentzian and reverse Khovanskii--Teissier conditions.
\end{example}

We now prove Corollary~\ref{cor:monomial}.
While the result follows by taking complete intersections and a limit of ample classes, we include its proof from the viewpoint of linear operators preserving volume polynomials.

\begin{proof}[Proof of Corollary~\ref{cor:monomial}]
When $k=0$, the statement is Theorem \ref{thm introduction}. When $k>0$, consider the projective volume polynomial
\[
f(x_0,x_1,\dots,x_n,y_1,\dots,y_k)=\frac{1}{(d+k)!}\deg\left(\left(x_0\alpha+\sum_{i=1}^n x_i\beta_i+\sum_{j=1}^k y_j\gamma_j\right)^{d+k}\right).
\]
By \cite[Corollary~1.4]{LinearOperator}, $\partial_{y_1}\cdots\partial_{y_k}f$ is also a projective volume polynomial. Consider the multiset $U$ on the indices of these variables defined by $U(x_i)=d$ for $0\leq i\leq n$ and $U(y_j)=0$ for $1\leq j\leq k$. 
By Theorem~\ref{thm:truncation} and continuity, the polynomial
\[
g:=
(\partial_{y_1}\cdots\partial_{y_k}f)_{\le U}
=
\left.
\partial_{y_1}\cdots\partial_{y_k}f
\right|_{y_1=\cdots=y_k=0}
\]
is also a projective volume polynomial. Applying Corollary \ref{thm:intro-volume-polynomials} to $g$ at $(1,0,\dots,0)$, we derive the desired inequality. 
\end{proof}

\begin{remark}\label{rmk:covolume}
We note that the key step in Corollary \ref{cor:monomial} used the fact that $\partial_{y_1}\cdots\partial_{y_k}f$ is a projective volume polynomial when $f$ is a projective volume polynomial. This follows from the fact that the monomial $\partial_{y_1}\cdots\partial_{y_k}$ is a covolume polynomial. In general, by \cite[Corollary~1.4]{LinearOperator} we may apply any covolume polynomial to a volume polynomial, and the coefficients of the resulting polynomial still satisfy the inequality of Theorem \ref{thm introduction}. See \cite[Remark~1.8]{LinearOperator} for more details.
\end{remark}

\begin{proof}[Proof of Corollaries~\ref{cor:discriminant main} and ~\ref{cor: mixed volumes}] 
By Corollary~\ref{cor:projective strong}, any non-strict inequality of intersection numbers of nef real divisor classes on smooth projective varieties also applies to mixed volumes of convex bodies and mixed discriminants of positive semidefinite matrices. Thus, Corollary~\ref{cor:monomial} implies Corollaries~\ref{cor:discriminant main} and ~\ref{cor: mixed volumes}.
\end{proof}

We may now deduce Theorem~\ref{thm introduction} for nef $(1,1)$-classes on a compact K\"ahler manifold $X$ of dimension $d$.

\begin{proof}[Proof of Theorem~\ref{thm introduction} for nef $(1,1)$-classes on compact K\"ahler manifolds]
Applying Proposition~\ref{thm:disc=>Kahler} to  Corollary \ref{cor:discriminant main} with $k=0$, we obtain Theorem~\ref{thm introduction} for nef $(1,1)$-classes on compact K\"ahler manifolds.
\end{proof}

\section{Applications}\label{sec:applications}

\subsection{Negative correlation for matroids}\label{sec:matroids}

We use Theorem~\ref{thm introduction} to prove the following extension of Theorem~\ref{thm:intro-basis-partition}. 
Let $M$ be a matroid on a finite set $E$.
For any subset $S \subseteq E$, let $\Pr(S)$ be the probability that a uniformly random basis of $M$ contains $S$.

\begin{theorem}\label{thm:matroid}
Let $M$ be a matroid linearly realizable over a field. 
For any $r$-cover $I=(I_1,\ldots,I_m)$ of a basis $B$ of $M$,
we have
\[
 \Pr(B)^r \leq \Pr(I_1)\Pr(I_2)\cdots\Pr(I_m).
\]
\end{theorem}

Since the partitions of $B$ are precisely the $1$-covers of $B$, Theorem~\ref{thm:intro-basis-partition} is the special case $r=1$ of Theorem~\ref{thm:matroid}.

\begin{proof}
Let $d$ be the rank of $M$, and suppose that $M$ is linearly realizable over a field $\K$. Without loss of generality, we may assume that $d$ is positive and $\K$ is algebraically closed. 
Since $M$ is linearly realizable over $\K$, there is a $d$-dimensional  linear subspace $V \subseteq \K^E$ such that a subset $S \subseteq E$ is independent in $M$ if and only if the projection $V \to \K^S$ is surjective.

Let $Y_V$ be the closure of $V$ in $(\PP^1)^E$.
The projective variety $Y_V$, called the \emph{arrangement Schubert variety} of $V$, plays a central role in the theory of intersection cohomology for matroids; see \cite[Section 1.3]{SHCG} and \cite[Section~4]{HuhWang17} for overviews. 
For any $e\in E$, let $\beta_e$ be the nef divisor class on $Y_V$ obtained as the pullback of the class of a point under the composition
\[
Y_V\longrightarrow (\PP^1)^E \longrightarrow \PP^1,
\]
where the first map is the inclusion and the second is the $e$-th projection. Then, by \cite[Lemma~15]{HSW22}, for any $d$-element subset $S\subseteq E$, we have
\[
\left(\prod_{e\in S}\beta_e\right)=\begin{cases}1&\text{when $S$ is a basis of $M$},\\0&\text{otherwise.}\end{cases}
\]
In particular, we have $(\beta^{[B]})=(\beta^{B})=1$.

Let $\alpha$ be the sum of all $\beta_e$ for $e \in E$. Since $\beta_e^2=0$ for any $e \in E$,  the displayed formula for the intersections of $\beta_e$ implies that
\[
(\alpha^{[d]})=(\alpha^d)/d!=(\text{the number of bases of $M$}).
\]
Similarly, for any independent set $S$ of $M$, we have
\[
(\beta^{[S]}\alpha^{[d-|S|]})=(\beta^{S}\alpha^{[d-|S|]})=(\text{the number of bases of $M$ containing $S$}).
\]
Thus, the conclusion follows from Theorem~\ref{thm introduction} applied to $\alpha$ and $\beta=(\beta_e)_{e \in E}$ on $Y_V$.
\end{proof}

As observed in the introduction, the multiple symmetric basis-exchange property implies that the bipartition case of Theorem~\ref{thm:matroid} holds for all matroids.
Another simple special case arises when $d \ge 2$ and $I$ is the $(d-1)$-cover of $B$ consisting of the $d$ sets $B\setminus e$ for $e \in B$.
In this case, Theorem~\ref{thm:matroid} follows from the \emph{bijective basis-exchange property} of Brualdi \cite{Brualdi}:
\begin{quote}
\emph{For any matroid $M$,  any bases $B$ and $B'$ of $M$, there is a bijection $\sigma:B \to B'$ such that $(B \setminus e) \cup \sigma(e)$ is a basis for all $e \in B$.}
\end{quote}
Applying this property to each basis $B'$ of $M$, we obtain an injective map
\[
\big\{\text{bases of $M$}\big\} \hookrightarrow \prod_{e \in B} \big\{\text{bases of $M$ containing $B \setminus e$}\big\} .
\]
Consequently, the inequality of Theorem~\ref{thm:matroid} holds for every matroid $M$ in this case, regardless of whether $M$ is linearly realizable over a field. More generally, the \emph{partition basis-exchange property} of Greene and Magnanti \cite{GreeneMagnanti} implies that the inequality of Theorem~\ref{thm:matroid} holds for any matroid when $r=m-1$.

\begin{conjecture}
The inequality of Theorem \ref{thm:matroid} holds for any matroid. 
\end{conjecture}

More generally, we conjecture that the inequality of Theorem~\ref{thm introduction} holds for any matroid $M$ 
and for any nef divisor classes in the augmented Chow ring of $M$ in the sense of \cite{SS}.

\subsection{Quermassintegrals of projections}

As another application of Theorem~\ref{thm introduction}, we  prove the projection inequality proposed by Fradelizi--Giannopoulos--Meyer in \cite[Remark~3]{fradelizi2003mixed}.
For a convex body $L$ in a $q$-dimensional Euclidean space $F$, we set
\[
V_j(L)=\MV_F\left(L^{\times j},\mathbb B_F^{\times(q-j)}\right) \ \ \text{for} \ \ 0\leq j\leq q,
\]
where $\MV_F$ denotes the mixed volume computed in $F$ and $\mathbb B_F$ is the Euclidean unit ball in $F$. Thus, for example, we have
\[
V_q(L)=\vol_F(L) \ \  \text{and} \ \  V_0(L)=\vol_F(\mathbb B_F).
\]
The case $\ell=0$ of the following theorem recovers \cite[Theorem~4.1]{fradelizi2003mixed}.

\begin{theorem}
Let $K\subseteq\RR^d$ be a convex body with nonempty interior, and let $\pi:\RR^d\to F$ be the orthogonal projection onto a nonzero $p$-dimensional subspace $F \subseteq \RR^d$.
For all integers $0\leq\ell\leq k\leq p$, we have
\[
\binom{d-p+k-\ell}{d-p} \frac{V_{d-k}(K)}{V_{d-\ell}(K)}\geq\frac{V_{p-k}(\pi K)}{V_{p-\ell}(\pi K)}.
\]
\end{theorem}

\begin{proof}
The cases $k=\ell$ and $p=d$ are equalities, so suppose that $\ell<k$ and $p<d$. Choose an orthonormal basis $e_1,\dots,e_d$ of $\RR^d$ such that $F$ is spanned by $e_1,\dots,e_p$. 
Let $\mathscr{E}_i$ be the unit segment in the $e_i$-th direction.
Let $\mathbb B$ be the Euclidean unit ball in $\RR^d$ and set
\[
Q_j=\MV\left(\mathscr{E}_{p+1},\dots,\mathscr{E}_d,K^{\times(p-j)},\mathbb B^{\times j}\right) \ \ \text{for $0\leq j\leq p$.}
\]
The mixed-volume formula for orthogonal projections \cite[Section 5.3]{Schneider:BM} gives
\[
Q_j=\frac{p!}{d!}V_{p-j}(\pi K).
\]
Indeed, adding a unit segment as one of the mixed-volume arguments amounts to projecting the remaining arguments  onto its orthogonal complement and dividing by the ambient dimension. Iterating this identity gives the factor $p!/d!$, and $\pi \mathbb B$ is the Euclidean unit ball $\mathbb B_F$.

We apply the bipartition case of Corollary \ref{cor: mixed volumes}, with $K$ as the reference body. The first part consists of the $d-p$ segments $\mathscr{E}_{p+1},\dots,\mathscr{E}_d$, and the second consists of $k-\ell$ copies of $\mathbb{B}$. 
As common auxiliary entries of $V$, insert $\ell$ copies of $\mathbb{B}$ and $p-k$ copies of $K$.
There are $d-p+k-\ell$ nonauxiliary arguments, so the resulting inequality is
\[
\binom{d-p+k-\ell}{d-p}V_{d-k}(K)Q_\ell\geq V_{d-\ell}(K)Q_k.
\]
Substituting the expressions for $Q_\ell$ and $Q_k$ and cancelling $p!/d!$ gives the desired result.
\end{proof}

\providecommand{\MR}[1]{}
\renewcommand{\MR}[1]{}

\bibliographystyle{amsalpha}
\bibliography{biblio}

\end{document}